\documentclass[11pt]{amsart}
\usepackage[english]{babel}
\usepackage[dvips]{graphics,color}
\usepackage{multicol}
\usepackage{amsmath}
\usepackage{amsfonts}
\usepackage{amssymb,latexsym,url}
\usepackage{psfrag,graphicx,epsfig}
\usepackage{enumerate}
\usepackage{amsthm}
\usepackage{version}

\usepackage{geometry}
\newtheorem{thm}{Theorem}[section]
\newtheorem{cor}[thm]{Corollary}

\newtheorem{prop}[thm]{Proposition}

\theoremstyle{definition}
\newtheorem{defn}[thm]{Definition}
\theoremstyle{remark}
\newtheorem{rem}[thm]{Remark}
\numberwithin{equation}{section}

\def \<{\langle}
\def \>{\rangle}

\newcommand{\R}{\mathbb{R}}
\newcommand{\N}{\mathbb{N}}

\newcommand{\B}{\mathbb{B}}

\begin{document}

\title{The minimum time function for the controlled Moreau's sweeping process}

\author[Giovanni Colombo]{Giovanni Colombo}
\address[Giovanni Colombo]{Universit\`a di Padova, Dipartimento di Matematica and I.N.d.A.M research group, via Trieste 63, 35121 Padova, Italy}
\email{colombo@math.unipd.it}

\author[Michele Palladino]{Michele Palladino}
\address[Michele Palladino]{Penn State University, Department of Mathematics, University Park, Pa. 16802, U.S.A.}
\email{mup26@psu.edu}

\thanks{This work was partially supported by 
the European Union under the 7th Framework Programme ``FP7-PEOPLE-2010-IT'', Grant agreement number 264735-SADCO}



\date{\today}
\begin{abstract}
Let $C(t)$, $t\geq0$ be a Lipschitz set-valued map with closed and (mildly non-)convex values and $f(t, x,u)$ be a map,
Lipschitz continuous w.r.t. $x$.
We consider the problem of reaching a target $S$ within the graph of $C$ subject to the differential inclusion
\[
(\star)\qquad \dot{x} \in -N_{C(t)}(x) + G(t,x)
\]
starting from $x_{0}\in C(t_{0})$ in the minimum time $T(t_{0},x_{0})$. The dynamics $(\star)$ is called a perturbed sweeping (or Moreau) process. 
We give sufficient conditions for $T$ to be finite and continuous and characterize
$T$ through Hamilton-Jacobi inequalities. Crucial tools for our approach are characterizations of weak and strong flow invariance of a set $S$
subject to $(\star)$. Due to the presence of the normal cone $N_{C(t)}(x)$, the right hand side of $(\star)$ contains implicitly the state constraint
$x(t)\in C(t)$ and is not Lipschitz continuous with respect to $x$. 
\end{abstract}
\maketitle

\section{Introduction}
%
%
%
The sweeping process is the time dependent evolution inclusion 
\begin{equation} \label{Sweeping}
\dot{x}(t) \in -N_{C(t)}(x(t)), \quad t\geq0,
\end{equation}
where $N_{C(t)}(x)$ denotes a suitable normal cone to $C(t)$  at $x$ (details will be provided in the sequel).
This dynamics was introduced in the 1970's  by J. J. Moreau to model quasi static evolution processes
subject to unilateral constraints (see \cite{Mo74}, \cite{Mo77}). 
Here the state trajectory $x(\cdot)$ is supposed to belong to a (possibly) infinite dimensional Hilbert space and $C(\cdot)$ is  
a moving convex set, Lipschitz continuous with respect to time. Several existence (and possibly uniqueness) results for 
the Cauchy problem associated to various generalisations of \eqref{Sweeping} were obtained by several authors 
(see \cite{marq,KM,CT} and references therein). 
In particular existence and uniqueness properties of solutions for the problem
\begin{equation} \label{Sweeping_pert}
\left\{
\begin{aligned}
\dot{x}(t) &\in -N_{C(t)}(x(t))+f(x(t)), \quad t\geq0.
\\
x(0)&=x_{0}\in C(0)
\end{aligned}
\right.
\end{equation}
are well known for a Lipschitz continuous $f$, even in the case in which $C$ 
is mildly non-convex (actually, uniformly prox-regular,
see \cite[Theorem 4.4]{Th03}). 
There has been an increasing interest in a number of variants of the sweeping process as a model, e.g., for some electric circuits \cite{ABB}, 
for crowd motion \cite{venel}, for hysteresis \cite{krejci}, and as a tool for identification of parameters \cite{brprta} for some mechanical systems
subject to unilateral constraints. On the other hand, a natural modification of such given models concerns 
the case in which the behavior of the whole system is modified in order to satisfy some expected performance requirements.
Following the latter direction, control theory of sweeping process comes into the picture. In general, the control may appear  
on the moving set $C$ (see \cite{CHHM,CHHM1,CHHM2} for the case where $f\equiv 0$ and $C$ is a moving polyhedron to be determined)
and/or on the perturbation $f$ (see, e.g., \cite{TCB1,TCB2} and particularly \cite{brokre}, where a complete discussion of the maximum principle appears,
in the case where $C$ is fixed, smooth, and strictly convex). 
In this paper we focus on the case where 
$C$ is a given, nonconstant, multifunction and $f$ is actually $f(x,u)$, $u\in U$ being the control. 

Observe that the state constraint $x(t)\in C(t) $ is implicitly satisfied by any solution of  
\eqref{Sweeping_pert}, since the normal cone $N_{C(t)}(x)$ is empty whenever $x\notin C(t)$ and it is equal to 
$\{ 0\}$ whenever $x\in \mathrm{int}\,C(t)$. 
In particular, the trajectory $x(t)$ is driven by $f(x(t))$ 
as long as it lies in the interior of $C(t)$, while, when on the boundary, the normal cone component of $\dot{x}(t)$ may be nonvanishing  
in a way that actually forces the state trajectory to remain in $C(t)$ for all $t$. Therefore \eqref{Sweeping_pert} can be 
regarded as a state constrained problem where the constraint is ``active in the dynamics''. 
In the case when $C(\cdot)\equiv C$ such a problem is referred to as having \textit{reflecting boundary} \cite{lisznit,Se05}. 

The aim of the present paper is studying the minimum time problem to reach a target subject to   
\begin{equation}\label{PSP_intro}
\left\{ \begin{array}{l}
\dot{x}(t)\,\in \,  -N_{C(t)}(x(t)) + f(x(t),u(t))\quad \mathrm{a.e.}\; t > t_{0},\quad u(t)\in U,
\\
x(t_0)=x_{0}\in C(t_{0})\;,
\\
x(t)\in C(t)\quad \mathrm{for}\;\mathrm{all} \quad t > t_{0}.
\end{array}\right.
\end{equation}
In particular, we focus on the minimum time function $T(\cdot,\cdot)$ as depending on both initial conditions. 
We give sufficient conditions for $T$ to be continuous (with an explicit modulus of continuity) 
and we characterise $T$ as the unique viscosity solution of a suitable set of Hamilton-Jacobi inequalities. 
A crucial role  in our analysis is played by the Hamiltonian characterisations of weak and strong flow invariance 
for a closed set $K$ subject to \eqref{PSP_intro}. Here we follow a set-valued approach to Hamilton-Jacobi 
characterisation which was designed by several authors (see the historical notes to Chapter 12 in \cite{vinter}), starting from 
Frankowska \cite{Fr} and, specifically for the minimum time, from Wolenski-Zhuang \cite{WoZh}. 
According to this approach, the epigraph and the hypograph of $T$ are shown to be, respectively, weakly and 
strongly invariant for a suitable augmented dynamics and the main point of the proof is actually giving the characterizations
of weak and strong invariance through Hamiltonian inequalities. The main difficulty in this framework is the 
lack of Lipschitz continuity of the right hand side of \eqref{PSP_intro}, which has just closed graph and it is not 
locally Lipschitz continuous with respect to $x$. In order to overcome this difficulty, 
one may observe that, using the close-to-convexity assumptions on $C(t)$, the right hand 
side of \eqref{PSP_intro} is one-sided Lipschitz (see \cite{Do91}).  
This condition is known to be  a good substitute of Lipschitz continuity in particular with 
respect to the characterisation of strong invariance (see \cite{DoRiWo}), which in general is a difficult issue to solve.
Here, our main contribution is Theorem \ref{SI}, where, taking into account the particular structure of 
\eqref{PSP_intro}, a characterisation of strong invariance by means of a \textit{single} 
Hamiltonian inequality is given. This is in contrast with the (albeit more general) result contained in \cite{DoRiWo},
where the necessary and the sufficient condition for strong invariance use \textit{different} Hamiltonians.
The main result of the paper is Theorem \ref{main}, where the Hamilton-Jacobi characterisation is provided. 
In Section \ref{sec:cont} a simple sufficient condition for the continuity of $T$ is given, while 
Section \ref{Ex} contains two examples where the minimum time function is computed through a verification argument.
Finally, Section \ref{sec:comp} is devoted to a comparison of our main result with the 
HJ-characterisation of the minimum time $T$ for classical state constrained problems,

We mention  that the case of a Mayer problem with $C(t)\equiv C$ was treated by Serea \cite{Se05} via different techniques. 
However, while the reflecting boundary problem solution can be characterised through the projection operator 
onto the tangent cone to $C$, it turns out that this is not the case when we consider a moving set $C$. 
This fact will become evident in Example 1, Section \ref{Ex}. Partial results, in the framework of Hamilton-Jacobi theory,
for the controlled sweeping process were also obtained in \cite{CMMR}. In particular, the value function
of an optimal control problem for \eqref{PSP_intro} involving an integral functional with finite horizon is proved to be
a viscosity subsolution of a suitable PDE of Hamilton-Jacobi type.

Finally we remark that all statements will be formulated in terms of $G(x)=f(x,U)$, $G$ Lipschitz and compact, convex valued.
Actually, it is well known that under such assumptions the two approaches are equivalent (see, e.g., \cite[?]{AuFr}). 
An explicit dependence of $G$ on $t$ will also be allowed.
 
\section{Preliminaries}
We first fix some notations and then proceed to define some concepts which will be used throughout the paper.

The unit ball in $\mathbb{R}^n$ is denoted by $\mathbb{B}$. The positive cone generated by a set $A\subset \R^n$ is
\[
\text{cone}\, A = \{ tx : x\in A, t\ge 0\}  
\]
and we set also $\| A\| = \sup\{ \| x\|: x\in A\}$. For a function $f:A\to \R$ we will consider both the \textit{epigraph} and the
\textit{hypograph} of $f$, denoted by $\text{epi}(f)$ and $\text{hypo}(f)$, respectively. The graph of a (possibly set-valued) map
$F:A\leadsto \R$ is named ${\rm graph}(F) := \{ (x,y)\in A\times\R : y\in F(x)\}$. 
The usual (semi)continuity as well as Lipschitz continuity concepts
for a set-valued map will be used (see, e.g., \cite[Section 1.1]{AuCe}). A map $F:[0,+\infty)\times A\leadsto \R^n$ is said
\textit{almost upper semicontinuous} if it is jointly Lebesgue$\times$Borel measurable with respect to $(t,x)$ and u.s.c. with respect to
$x$ for a.e.~$t$. It is well known that this property is equivalent to a Scorza-Dragoni type of joint continuity, as stated, e.g., in \cite{DoRiWo}.

Let $C\subset\R^n$ be closed with boundary $\mathrm{bdry} C$.
Some concepts of nonsmooth analysis will be needed (see, e.g. \cite[Chapters 1 and 2]{CLSW} for more details). Given $x \in K$ and $v \in \R^n$,
we say that $v$ is a \emph{proximal normal} to $K$ at $x$, and denote as $v\in N_K(x)$,
provided that there exists $\sigma=\sigma(v,x)\ge 0$ such that
\[
v\boldsymbol{\cdot} (y-x)\le\sigma \|y-x\|^2,\quad \textrm{ for all }y\in K.
\]
If $K$ is convex, then $N_K(x)$ coincides with the normal cone of Convex Analysis.

Let $\Omega \subset \R^n$ be open and let $f: \Omega \rightarrow \R \cup\{+\infty\}$ be lower semicontinuous. 
The \textit{proximal subdifferential} $\partial_P f(x)$ of $f$ at a point $x$ of dom$(f)=\{ x:f(x)<+\infty\}$ is the set of vectors $v \in \R^n$ such that 
\[
(v,-1)\in N_{\text{epi}(f)}(x,f(x)).
\]
Symmetrically, for an upper semicontinuous $f$ the \textit{proximal superdifferential} $\partial^P f(x)$ of 
$f$ at a point $x\in\mathrm{dom}(f)$ is the set of vectors $v\in \R^n$ such that
\[
(-v,1)\in N_{\text{hypo}(f)}(x,f(x)).
\]
\textit{Prox-regular sets} will play an important role in the sequel. The definition was first given by Federer in \cite{Fe}, 
under the name of \textit{sets with positive reach}, and later studied by several authors (see the survey paper \cite{CT}).
\begin{defn} \label{posreach} 
Let $C\subset \R^n$ be closed and $r>0$ be given. We say that $C$ is $r$-prox-regular provided the inequality
\begin{equation}\label{ineqphi}
\langle \zeta,y-x\rangle\le \frac{1}{2r} \| \zeta\|\,\| y-x\|^2
\end{equation}
holds for all $x,y\in C$ and $\zeta\in N_{C}(x)$.
\end{defn}
In particular, every convex set is $r$-prox regular for every $r>0$.
Prox-regular sets enjoy several properties, including uniqueness of the metric projection and differentiability of the 
distance (in a suitable neighborhood) and normal regularity, which hold also true for convex sets, see, e.g. \cite{CT}.
We prove here a property that will be used in the sequel in the case in which $C$ depends on a parameter. It is a simple consequence of \eqref{ineqphi}.
\begin{prop}\label{almostusc}
Let $r>0$ be given and let $C: [0,T] \leadsto \R^n$ be a continuous set valued map such that, for all $t\in [0,T]$, 
$C(t)$ is $r$-prox regular. Then the map $F:\text{graph}(C)\leadsto \R^n$, $F(t,x)=N_{C(t)}(x)$ has closed graph. Consequently, $\tilde{F}(t,x):=
F(t,x)\cap\B$ is upper semicontinuous.
\end{prop}
\begin{proof}
Let the sequence $\{ (t_j,x_j,\zeta_j): (t_j,x_j)\in \text{graph}(C), \zeta_j\in N_{C(t_j)}(x_j), n \in \N\}  $ be given and suppose that
$(t_j,x_j,\zeta_j)\to (t,x,\zeta)$. By assumption, $x\in C(t)$. We wish to prove that $\zeta\in N_{C(t)}(x)$, namely that
\begin{equation}  \label{closgr}
\zeta \boldsymbol{\cdot} (y-x) \le \frac{1}{2r} \|\zeta\| \, \| y-x\|^2
\end{equation}
for all $y\in C(t)$.

Fix $y\in C(t)$. By the lower semicontinuity of $C$, for each $j\in\N$ there exists $y_j\in C(t_j)$ such that $y_j\to y$. Then we have
\[
\begin{split}
\zeta \boldsymbol{\cdot} (y-x) &= (\zeta-\zeta_j)\boldsymbol{\cdot} (y-x) + \zeta_j \boldsymbol{\cdot} (y_j-x_j) + 
\zeta_j \boldsymbol{\cdot} \big (y-y_j - (x_j-x)\big)\\
&\le \frac{1}{2r} \| \zeta_j\| \| y_j -x_j\|^2 + o(1). 
\end{split}
\]
By passing to the limit, the proof of \eqref{closgr} is concluded. The upper semicontinuity of $\tilde{F}$
is obtained by standard arguments of set-valued analysis (see, e.g., Corollary 1, p. 42, in \cite{AuCe}).
\end{proof}
A concept, related to prox-regularity (of the complement of a set) but weaker, is the $r$-\textit{internal sphere condition}. This amounts to requiring that
\eqref{ineqphi} holds true for \textit{some} nonvanishing normal vector $\zeta\in N_C(x)$, for all $x\in C$.
It is well known (see, e.g., \cite[Proposition 2.2.2]{CS}) that this condition
is equivalent to the semiconcavity of the distance function $d_C$ to $C$ up to the boundary of $C$, namely that for each 
$x,y\in \overline{(\R^n\setminus C)}$ and $\zeta\in \partial^P d_C(x)$ it holds
\begin{equation}\label{semiconc}
d_C(y)\le d_C(x) + \zeta \boldsymbol{\cdot} (y-x) + \frac{1}{2r} \| y-x\|^2.
\end{equation}

\section{The Dynamics}\label{Dynamic}
The following standing assumptions will be valid throughout the paper.
\begin{itemize}
\item[$(H_{C})$:] $C:[0,\infty)\leadsto \R^{n}$ is a set-valued map with the following properties:
	\begin{itemize}
	\item[$H_{C1})$:] for all $t\geq0$, $C(t)$ is nonempty and compact and there exists $r>0$ such that $C(t)$ is $r$-prox regular.
	\item[$H_{C2}):$] $C$ is Lipschitz continuous with constant $L_{C}$ (namely $d_{H}(C(t),C(s))\le L_C |t-s|$ for all $s,t\in [0,T]$, where
$d_{H}$ denotes the Hausdorff distance between subsets of $\R^n$).
	\end{itemize}
\item[$(H_{G})$:] $G:[0,\infty)\times\R^{n}\leadsto \R^{n}$ is a set-valued map with nonempty, closed and convex values, such that
	\begin{itemize}
	\item[$H_{G1})$:] there exists $M>0$ such that 
	$$ \|G(t,x)\| \leq M \qquad \forall \, (t,x)\in [0,\infty) \times \mathbb{R}^{n};$$
	\item[$H_{G2})$:] $G$ is Lipschitz continuous with constant $L_{G}$.
	\end{itemize}
\end{itemize}

The following result, a special case of a theorem due to Thibault \cite[Proposition 2.1 and Theorem 3.1]{Th03}, will be invoked 
to guarantee the well-posedness of the relevant Cauchy problem.

\begin{thm}\label{Thibault}
Let the set-valued maps $C$ and $G$ satisfy assumptions $(H_{C})$ and $(H_{G})$, respectively.
Then the Cauchy problem
\begin{equation}\label{PSP}
\left\{ \begin{array}{l}
\dot{x}(t)\,\in \,  -N_{C(t)}(x(t)) + G(t,x(t))\quad \mathrm{a.e.}\; t > t_{0},
\\
x(t_{0})=x_{0}\in C(t_{0})\;,
\\
x(t)\in C(t)\quad \mathrm{for}\;\mathrm{all} \quad t > t_{0},
\end{array}\right.
\end{equation}
admits solutions, and the solution set is closed w.r.t. the uniform convergence. 
Moreover, if $x$ is a solution of \eqref{PSP}, then it is a solution of the (unconstrained) problem
\begin{equation}\label{SP'}
\left\{ \begin{array}{l}
\dot{x}(t)\,\in \,  -(L_{C}+M)\,\partial d_{C(t)}(x(t))+ G(t,x(t))\quad \mathrm{a.e.}\; t > t_{0},
\\
x(t_{0})=x_{0}\in C(t_{0})\;.
\end{array}\right.
\end{equation}
\end{thm}

As a side remark, we mention the fact that for each $v\in L^{1}_{loc}([t_{0},\infty), \R^{n})$ the Cauchy problem 
\begin{equation}\label{SPU}
\left\{ \begin{array}{l}
\dot{x}(t)\,\in \,  -N_{C(t)}(x(t)) + v(t)\quad \mathrm{a.e.}\; t > t_{0},
\\
x(t_{0})=x_{0}\in C(t_{0})\;,
\\
x(t)\in C(t)\quad \mathrm{for}\;\mathrm{all} \quad t > t_{0},
\end{array}\right.
\end{equation}
admits a unique solution. This follows from the (hypo)monotonicity of the normal cone to a prox-regular set and Gronwall's 
Lemma using well established arguments.

\section{Weak and Strong Invariance} \label{INV}

Suppose $C, G$ be satisfying  $(H_{C})$ and $(H_{G})$, respectively. Let  $K:[0,\infty) \leadsto \R^{n}$ be a set-valued 
map with closed graph such that $K(t)\subseteq C(t)$ for all $t\geq0$. We recall that closedness of $\mathrm{graph}(C)$ is a consequence of $(H_{C})$.

\begin{defn} {\bf (Weak invariance).} We say that  $K$ is weakly invariant with respect to the perturbed sweeping dynamics $-N_{C(t)}(x(t))+G(t,x(t))$ 
if and only if, for all $(t_{0},x_{0})\in \mathrm{graph}(C)\cap \mathrm{graph}(K)$, the Cauchy problem \eqref{PSP} admits, for some $T>0$, at least a solution 
$x:[t_{0}, T]\rightarrow \R^{n}$ such that $x(t)\in K(t)$ for all $t\in [t_{0}, T)$. 
\end{defn}

\begin{defn} {\bf (Strong invariance).} We say that  $K$ is strongly invariant with respect to the perturbed sweeping dynamics 
$-N_{C(t)}(x(t))+G(t,x(t))$ if and only if, for all $(t_{0},x_{0})\in \mathrm{graph}(C)\cap \mathrm{graph}(K)$, $T>0$ and $x:[t_{0}, T]\rightarrow \R^{n}$ 
solution to the Cauchy problem \eqref{PSP}, we have $x(t)\in K(t)$ for all $t\in [t_{0}, T]$. 
\end{defn}

Necessary and sufficient conditions for weak (also known as viability) and strong invariance are well known in the particular case 
where $C(t)\equiv \R^{n}$ for all $t$ (i.e., the normal part on the right hand side of \eqref{PSP} vanishes). Such a characterisation is 
stated both through tangency conditions (see, e.g., \cite{Aubin}) and through inequalities involving normal vectors (see, e.g., \cite{CLSW}). 
Our results will be expressed in terms of normals since we are interested in Hamilton-Jacobi inequalities.

We restrict our attention to the case in which $C$ is constant and state the relevant characterisation for an augmented 
system in order to be ready for the application to epi/hypograph of the minimum time function $T$,
which will appear in Section \ref{H-J_ineq_sec}.  
The first result deals with weak invariance and it is largely based on the existing theory.

\begin{thm} \label{wih}
Let the set-valued maps $C$ and $G$ satisfy assumptions $(H_{C})$ and $(H_{G})$, respectively. 
Take $K\subseteq \R^{n}$ closed and such that 
\[
C(t) \cap K \neq \emptyset \qquad \forall\; t\geq 0.
\]
For $t_{0}\geq 0$ and $x_{0} \in K \cap C(t_{0})$, consider the Cauchy problem
\begin{equation}\label{APSP}
\left\{ \begin{array}{l}
\dot{\tau}(t) = 1,
\\
\dot{x}(t)\,\in \,  -N_{C(\tau(t))}(x(t)) + G(\tau(t), x(t))\quad \mathrm{a.e.}\; t > 0,
\\
(\tau(0), x(0))= (t_{0}, x_{0})\;.
\end{array}\right.
\end{equation}
Set $\mathcal{K}:= [0,\infty) \times K$. Then $\mathcal{K}$ is weakly invariant for \eqref{APSP} if and only if, for every 
$(\tau,x)\in \mathrm{graph}(C)\cap \mathcal{K}$, we have
\begin{equation}
\label{Hmeno}
\min_{v\in\{0\}\times \{ -N_{C(\tau)}(x)\cap(L+M)\B\}} v\boldsymbol{\cdot} p + \min_{v\in \{1\} \times G(\tau,x)} v\boldsymbol{\cdot} p \leq 0 \qquad 
\forall \; p\in N^{P}_{\mathrm{graph}(C)\cap \mathcal{K}}(\tau,x). 
\end{equation}
\end{thm}
\begin{proof}
According to Theorem \ref{Thibault}, \eqref{APSP} is equivalent to the Cauchy problem with bounded right hand side  
\[ 
\left\{ \begin{array}{l}
\dot{\tau}(t) = 1,
\\
\dot{x}(t)\,\in \,  -(L+M)\partial d_{C(\tau(t))}(x(t)) + G(\tau(t), x(t))\quad \mathrm{a.e.}\; t > 0,
\\
(\tau(0), x(0))= (t_{0}, x_{0})\;.
\end{array}\right.
\]
Applying the characterisation of weak invariance proved in \cite[Theorem 2.10]{CLSW} then \eqref{Hmeno} follows. The proof is concluded.
\end{proof}

The second result of this section deals with strong invariance and is new. Indeed, observe that the right hand side of \eqref{APSP} 
does not satisfy the one-sided Lipschitz continuity assumption which is imposed in \cite{DoRiWo} in relation with 
strong invariance. Here, taking advantage of the special structure of the dynamics we consider, we are able to characterize 
strong invariance by means of a {\em single} Hamiltonian inequality. In the more general strong invariance result contained in \cite{DoRiWo},
instead, the authors use two different inequalities, one for the necessity condition and one for the sufficiency condition.

The notations of Theorem \ref{wih} are adopted also in the following result.

\begin{thm}
\label{sih} Let the assumptions and notations of Theorem \ref{wih} hold true. Then $\mathcal {K}\cap \mathrm{graph}(C) $ is strongly 
invariant for \eqref{APSP} if and only if,  for every $(\tau,x)\in \mathrm{graph(C)}\cap \mathcal{K}$, we have
\begin{equation}
\label{Hpiu}
\min_{v\in\{0\}\times \{ -N_{C(\tau)}(x)\cap(L+M)\B\}} v\boldsymbol{\cdot} p + \max_{v\in \{1\} \times G(\tau,x)} v\boldsymbol{\cdot} p \leq 0 \qquad 
\forall \; p\in N^{P}_{\mathrm{graph}(C)\cap \mathcal{K}}(\tau,x). 
\end{equation}
\end{thm}

\begin{proof} \textbf{(Necessity)}. Fix $(\bar{\tau},\bar{x})\in \mathrm{graph}(C) \cap \mathcal{K}$ and  $\bar{p} \in 
N^{P}_{ \mathrm{graph}(C)\cap\mathcal{K}}(\bar{\tau},\bar{x})$. Take $\bar{v}\in \{1\}\times G(\bar{\tau}, \bar{x})$ such that
\[
\bar{v}\boldsymbol{\cdot} \bar{p} = \max_{v\in  \{1\}\times G(\bar{\tau},\bar{x}) } v\boldsymbol{\cdot} \bar{p}.
\]
Let $\phi : \mathrm{graph}(C)\rightarrow \R^{n}$ be a continuous selection from $G(\tau,x)$ such that 
$\phi(\bar\tau,\bar{x})=\bar{v}$ (such a selection exists by stabdard arguments of set-valued analysis, see, e.g., 
Theorem 1, p. 82, and the beginning of the proof of Corollary 1, p. 83, in \cite{AuCe}. 
Recall that the Cauchy problem
\begin{equation}\label{RSP}
\left\{ \begin{array}{l}
\dot{\tau}(t) = 1
\\
\dot{x}(t)\,\in \,  -N_{C(\tau(t))}(x(t)) + \phi(\tau(t), x(t))\quad \mathrm{a.e.}\; t > 0,
\\
(\tau(0), x(0))= (t_{0}, x_{0})\;
\end{array}\right.
\end{equation}
admits solutions. From the strong invariance of $\mathcal{K}$ with respect to \eqref{APSP}, $\mathcal{K} $ is also weakly 
invariant with respect to \eqref{RSP}. Thus we obtain from Theorem \ref{wih} that
\[
\min_{v\in\{0\}\times \{ -N_{C(\tau)}(x)\cap(L+M)\B\}} v\boldsymbol{\cdot} p + \phi(\tau,x)\boldsymbol{\cdot} p \leq 0 \qquad \forall \; 
p\in N^{P}_{\mathrm{graph}(C)\cap \mathcal{K}}(\tau,x)
\]
for all $(\tau,x)\in \mathrm{graph}(C)\cap \mathcal{K}$. In particular, by taking $(\tau,x)=(\bar{\tau}, \bar{x})$ and $p=\bar{p}$, we obtain
\[
\min_{v\in\{0\}\times \{ -N_{C(\bar{\tau})}(\bar{x})\cap(L+M)\B\}} v\boldsymbol{\cdot} \bar{p} + \max_{v\in \{1\} \times 
G(\bar{\tau},\bar{x})} v\boldsymbol{\cdot} \bar{p} \leq 0 .
\]
Since $(\bar{\tau}, \bar{x})$ and $\bar{p}\in N^{P}_{\mathrm{graph}(C)\cap \mathcal{K}}(\bar{\tau}, \bar{x})$ 
were arbitrary chosen, the first part of the proof is completed.

\textbf{(Sufficiency)}. Take $T>0$ and a solution $(\tau,x):[0,T]\rightarrow \R^{1+n}$ of \eqref{APSP} with $(t_{0},x_{0}) 
\in \mathcal{K}\cap\mathrm{graph}(C)$. We aim at proving that $(\tau,x)(t)\in \mathcal{K}$ for each $t\in [0,T]$.

To this aim, observe first that the structure of the right hand side of \eqref{APSP} implies through standard selection theorems
(see, e.g., Proposition 4, p. 43, and Theorem 1, p. 90, in \cite{AuCe}) that there exist measurable functions 
$\xi, g : [0, T]\rightarrow \R^{n}$ such that, for a.e. $t\in [0,T]$,
$$ \xi(t)\in N_{C(\tau(t))}(x(t))\cap (L_{C}+M)\B,\qquad g(t)\in G(\tau(t),x(t)),$$
$$ \dot{x}(t) = -\xi(t)+ g(t).$$
Define 
\begin{equation}\label{defL}
L:= L_{C}+M,
\end{equation}
and, for $(\tau,x)\in \mathrm{graph}(C)$, 
\begin{equation} \label{def_F}
F(\tau,x) := \{  1\} \times \big(-N_{C(\tau)}(x)\cap L \B+G(\tau,x)\big),
\end{equation}
and, for a.e. $t\in [0,T]$,
\begin{equation} \label{Gtilde_def}
\begin{split}
\tilde{G}(t,\tau,x)&:= \Big\{ v\in F(\tau,x) :  \big((\dot\tau(t), \dot{x}(t))-v\big)\boldsymbol{\cdot} \big((\tau(t),x(t))-(\tau,x)\big) \leq \frac{2L}{r}\|x(t)-x\|^{2}+
\\
 &\qquad + \frac{2LL_C^2}{r}|\tau(t)-\tau|^{2}+L_{G}\|(\tau(t),x(t))-(\tau,x)\|\, \|x(t)-x\| +
 2LL_{C}|\tau-\tau(t)|\Big\}
 \end{split}
\end{equation}
We claim that the set-valued map $\tilde{G}$ has the following properties: 
\begin{itemize}
\item[ (i) ] $\emptyset\neq \tilde{G}(t,\tau, x) \subseteq F(\tau,x)$ for a.e. $t\in [0,T]$, for all $(\tau,x)\in \mathrm{graph}(C)$;
\item[ (ii) ] $\min_{v\in \tilde{G}(t,\tau,x)} v\boldsymbol{\cdot} p \leq 0 $ for all $(\tau,x)\in \mathrm{graph}(C)\cap \mathcal{K}$, 
a.e. $t\in [0,T]$, $p\in N^{P}_{\mathrm{graph}(C)\cap \mathcal{K}}(\tau,x)$;
\item[ (iii) ] $\tilde{G}$ is almost upper semicontinuous and takes values compact and convex sets. 
\end{itemize}
In order to see the non-emptiness of $\tilde{G}$, we establish first some inequalities. 
Fix $t\in[0,T]$ such that $\dot{x}(t)$ exists and take $(\tau,x)\in \mathrm{graph}(C)$. 
First of all, using the Lipschitz continuity of $G$, it follows that there exists $\bar{g}\in G(\tau,x)$ such that 
\begin{equation} \label{def_g}
\|g(t)-\bar{g}\| \leq L_{G} \| (\tau(t), x(t)) - (\tau,x)\| .
\end{equation}  
Using now the Lipschitz continuity of $C$, we can find $y\in C(\tau(t))$ and $z\in C(\tau)$ such that 
\begin{equation} \label{def_y}
\|x-y\| \leq L_{C} |\tau - \tau(t)|,
\end{equation}
\begin{equation} \label{def_z}
\|x(t) -  z \| \leq L_{C} |\tau(t)-\tau| .
\end{equation}
By taking into account the uniform prox-regularity of $C$ (see \eqref{ineqphi}) and \eqref{def_y}, we obtain
\begin{equation} \label{ineq_2}
\begin{split}
-\xi(t)\boldsymbol{\cdot} (x(t) -x) & = \xi(t)\boldsymbol{\cdot} (y - x(t)) + \xi(t)(x-y) \leq
\\
 & \leq \frac{1}{2r}\|\xi(t)\|\, \|x(t) - y\|^{2} + \|\xi(t)\| L_{C} | \tau - \tau(t)|
\\
 & \leq \frac{L}{2r}\|x(t)-y\|^{2}+LL_{C}|\tau - \tau(t)|.
\end{split}
\end{equation}
By arguing as to obtain \eqref{ineq_2}, for every $\xi \in N_{C(\tau)}(x)\cap L\B$, recalling  \eqref{ineqphi} and \eqref{def_z}
we obtain as well
\begin{equation} \label{ineq_3}
\xi \boldsymbol{\cdot} (x(t)-x) \leq \frac{L}{2r}\|x-z\|^{2} + LL_{C}|\tau(t)-\tau|.
\end{equation}
Set $\bar{v} = (1, -\xi + \bar{g})$, where $\xi$ is an arbitrary element of the set $N_{C(\tau)}(x)\cap L\B$. 
By taking into account \eqref{def_g}, \eqref{ineq_2}, and \eqref{ineq_3}, we obtain

\[
\begin{split}
\big( (\dot\tau(t), \dot{x}(t))&-\bar{v}\big)\boldsymbol{\cdot} \big( (\tau(t),x(t))-(\tau,x)\big) = 
\big((1, -\xi(t)+g(t))-(1,-\xi+\bar{g})\big)\boldsymbol{\cdot} (\tau(t)-\tau, x(t)-x)\\
&=(-\xi(t) + \xi +g(t) - \bar{g}) \boldsymbol{\cdot} (x(t) - x)
\\
&\leq \frac{L}{2r}\|x(t) - y\|^{2}+ LL_{C}|\tau-\tau(t)| + \frac{L}{2r}\|x-z\|^{2} + LL_{C}|\tau(t)-\tau| 
\\
&\qquad +L_{G}\|(\tau(t),x(t)) - (\tau,x))\| \, \|x(t)-x\|
\\
&\leq \frac{L}{r}\big( \|x(t) - x\|^{2} + \|x-y\|^{2}\big) + \frac{L}{r}\big( \|x(t)-x\|^{2}+\|x(t)-z\|^{2}\big) + 2LL_{C}|\tau - \tau(t)|
\\
&\qquad + L_{G}\|(\tau(t),x(t)) - (\tau,x)\| \, \|x(t)-x\| 
\\
&\leq \frac{2L}{r} \|x(t)-x\|^{2} + \frac{2LL^2_{C}}{r}|\tau(t)-\tau|^{2} + 2LL_{C}|\tau - \tau(t)|\\
&\qquad  + L_{G}\|(\tau(t),x(t)) - (\tau,x)\|\, \|x(t)-x\|,
\end{split}
\]

which shows that $\bar{v}\in \tilde{G}(t,\tau,x)$. 
Furthermore, property (ii) follows immediately from \eqref{Hpiu}, while (iii) is easily checked.
Thus the claim is confirmed.

Observe that in the above argument the only condition imposed on $\xi$ was $\xi \in N_{C(\tau)}(x)\cap L\B$. 
Then, given $p\in N_{\mathrm{graph}(C)\cap \mathcal{K}}(\tau,x),$ we can choose
$$ -\bar{\xi} \in \underset{v\in \{0\} \times \{ -N_{C(\tau)}(x)\cap L\B\}}{\mathrm{argmin}} v\boldsymbol{\cdot} p.$$
Therefore, by taking $\bar{g}$ as in \eqref{def_g}, we have
$$ \min_{v\in \tilde{G}(t,\tau,x)} v\boldsymbol{\cdot} p \leq (1, -\bar{\xi}+\bar{g})\boldsymbol{\cdot} p \leq \min_{v\in \{0\}\times 
\{-N_{C(\tau)(x)\cap L\B}\}} v\boldsymbol{\cdot} p + \max_{\{1\} \times G(\tau,x)} v\boldsymbol{\cdot} p \leq 0,$$
where last inequality holds true by hypothesis. 
Thus, by a known weak flow invariance result (see, e. g., Theorem 1 in \cite{DoRiWo}), the Cauchy problem
\[ 
\begin{cases}
(\dot{\tau}(t),\dot{x}(t)) \!\!&\in\; \tilde{G}(t,\tau(t),x(t))
\\
(\tau(0), x(0))\!\!&=\; (t_{0}, x_{0})\;
\end{cases}
\]
admits a solution $(\tilde{\tau}, \tilde{x}):[0,\tilde{T}]\rightarrow \R^{1+n}$ such that 
$(\tilde{\tau}(t),\tilde{x}(t))\in \mathrm{graph}(C)\cap \mathcal{K}$ for all $t\in [0, \tilde{T}]$. 
Set $T'= \min \{ T, \tilde{T}\}$. We claim that $(\tau(t), x(t))=(\tilde{\tau}(t), \tilde{x}(t))$ for all 
$t\in [0, T']$. Indeed, since $(\tilde{\tau}, \tilde{x})$ is a $\tilde{G}$-trajectory, by taking 
$v = (\dot{\tilde{\tau}}(t), \dot{\tilde{x}}(t))$ and $(\tau,x)=(\tilde{\tau}(t), \tilde{x}(t))$ in \eqref{Gtilde_def}, 
we obtain for a.e. $t\in [0,T']$
\begin{align}
\big((\dot{\tau}(t),  \dot{x}(t))   - (\dot{\tilde{\tau}}(t), \dot{\tilde{x}}(t))\big) & \boldsymbol{\cdot} \big((\tau(t),x(t)) - (\tilde{\tau}(t), \tilde{x}(t))\big)
 \leq\frac{2L}{r} \|x(t) - \tilde{x}(t)\|^{2} + \frac{2LL_{C}^2}{r}|\tau(t) - \tilde{\tau}(t)|\nonumber
\\
& \qquad + L_{G}\|(\tau(t),x(t)) - (\tilde{\tau}(t), \tilde{x}(t))\|\, \| x(t) - \tilde{x}(t) \|\\
& \qquad + 2LL_{C}|\tau(t) - \tilde{\tau}(t)|.\nonumber
\end{align}
Observe that $\tau(0)=\tilde{\tau}(0)=t_{0}$ and $\dot{\tau}(t) = \dot{\tilde{\tau}}(t)=1$, so that $\tau(t)\equiv \tilde{\tau}(t)$. 
Therefore the above inequality yields
$$(\dot{x}(t)-\dot{\tilde{x}}(t))\boldsymbol{\cdot} (x(t)-\tilde{x}(t)) \leq \left(\frac{2L}{r}+L_{G}\right)\|x(t)-\tilde{x}(t)\|^{2}.$$
Since $x(0)=\tilde{x}(0)=x_{0}$, it follows from Gronwall's Lemma that $x(t)=\tilde{x}(t)$ in $[0,T']$. 
If $T'<T$, we can repeat the same arguments starting from the point $(T',x(T'))$ in place of $(t_{0},x_{0})$. 
This proves that $(\tau(t),x(t))\in \mathcal{K}\cap \mathrm{graph}(C)$ for all $t\in[0,T]$. 
Since the trajectory $(\tau,x)$ was arbitrarily chosen, the proof is complete.
\end{proof}
\begin{rem}
Observe that in the Hamiltonian characterization of strong invariance \eqref{Hpiu}, the summand relative to $N_{C(t)}$ is minimised
and not maximised. Notice that the maximum of $p\boldsymbol{\cdot}\xi$ with $\xi\in -N_{C(t)}(x)$ is nonnegative, since  $0\in -N_{C(t)}(x)$. 
Therefore taking the maximum in condition (4.3) would imply that all the vectors of $G(t,x)$ point inward with respect to $C(t)$.
Taking the minimum in (4.3) reflects the features of the perturbed sweeping process solutions $x(\cdot)$ with a 
perturbation $G(t,x)$ which does not satisfy inward pointing conditions on $\partial C(t)$. In this case, a velocity $g\in G(t,x(t))$ 
could push the trajectory out from the set $C(t)$ 
and the action of a correcting term $\xi\in -N_{C(t)}(x)$ is required. The minimisation on $\xi\in -N_{C(t)}(x)$ 
takes into account the presence of such a correcting term.
\end{rem}

\section{Invariance properties and the value function}\label{DP}

The sweeping process, and its perturbed version as well, is by its very nature a time varying dynamics. 
For this reason we will consider the minimum time function as depending also on the initial time.
\begin{defn}
\label{Minimum_Time}
Let the multifunction $C$ and $G$ be given as in Section \ref{Dynamic} and consider the Cauchy problem \eqref{APSP}. 
Let $S\subset\mathrm{\R}^{n}$ be closed. The minimum time to reach $S$ subject to \eqref{APSP} is 
\begin{equation} \label{MT}
T(t_{0},x_{0})  = \inf \{ \alpha \geq 0 : x(\alpha)\in S, \; \text{$(\tau,x)$ solution of \eqref{APSP}} \}.
\end{equation}
\end{defn}
Observe that an obvious necessary condition for $T(t_{0},x_{0})$ to be finite is 
\begin{equation}\label{non-empty}
 S\cap C(t) \neq \emptyset \qquad \mathrm{for}\quad\mathrm{some}\quad t\geq t_{0}.
\end{equation}
It is well known that the epigraph of the value function satisfies some forward and backward invariance properties 
(see, e.g. \cite{Fr}, \cite{WoZh}). Since we are able to treat only continuos minimum time functions, we will substitute the 
backward invariance of the epigraph with the forward one for the hypograph. Our setting differs from the classical 
autonomous framework and so we state and prove in details the invariance properties of the value function. 
Essentially, we follow arguments of \cite[Section 3]{WoZh}.

We set
\begin{equation}\label{Epigraph}
E:=\mathrm{epi} (T)=\{ (t,x,\lambda): \; (t,x)\in \mathrm{graph}(C), \; \lambda\geq T(t,x) \}
\end{equation}
and
\begin{equation}
H:=\mathrm{hypo} (T)=\{ (t,x,\lambda): \; (t,x)\in \mathrm{graph}(C), \; \lambda\leq T(t,x) \}.
\end{equation}
\begin{prop} \label{WI}
Let the set-valued maps $C$ and $G$ satisfy assumptions $(H_{C})$ and $(H_{G})$, respectively, and let $S\subseteq \R^n$ be closed. 
Consider the minimum time function $T$ to reach $S$ subject to \eqref{APSP} and assume that 
$T:\mathrm{graph}(C)\rightarrow [0,\infty]$ is lower semicontinuous and not identically equal to $+ \infty$. 
Then $E$, the epigraph of $T$, is weakly invariant with respect to the  dynamics generated by the 
velocity set
\begin{equation} \label{Gamma_def}
\Gamma(\tau,x,\lambda):=\{ 1\} \times \{ -N_{C(\tau)}(x)+G(\tau,x) \} \times \{-1 \}
\end{equation}
for $(\tau,x)\in \mathrm{graph}(C)$, $\lambda\in \R$.
\end{prop}
\begin{proof}
Let $(t_{0},x_{0},\lambda_{0})\in E$. In particular, $T(t_{0}, x_{0})\leq \lambda_{0}<\infty$. 
By invoking Theorem \ref{Thibault}, it is not difficult to see that an optimal trajectory 
$(\bar{\tau},\bar{x})$ of \eqref{APSP} exists. 
Thus the optimality principle yields, for all $s\in[0,T(t_{0},x_{0})]$, 
$$T(t_{0}+s, \bar{x}(s))= T(t_{0}, x_{0}) - s \leq \lambda_{0}-s.$$
In other words the solution
$$\big(\tau(s),x(s), \lambda(s)\big)=\big(t_{0}+s, \bar{x}(s), \lambda_{0}-s\big)$$
of the Cauchy problem  $\big(\dot\tau(s), \dot{x}(s), \dot{\lambda}(s)\big)\in \Gamma\big(\tau(s),x(s),\lambda(s)\big)$, 
$\big(\tau(0),x(0), \lambda(0)\big)=\big(t_{0},x_{0}, \lambda_{0}\big)$ belongs to $E$ for all $s\in [0,T(t_{0},x_{0})]$. The proof is complete.
\end{proof}

\begin{prop}\label{SI}
Let the set-valued maps $C$ and $G$ satisfy assumptions $(H_{C})$ and $(H_{G})$, respectively, 
and let $S\subseteq \R^n$ be closed.
Consider the minimum time function $T$ to reach $S$ subject to \eqref{APSP} and assume that 
$T:\mathrm{graph}(C)\rightarrow [0,\infty]$ is upper semicontinuous and not identically equal to $+ \infty$. 
Then $H$, the hypograph of $T$, is strongly invariant with respect to the  dynamics generated by the 
velocity set \eqref{Gamma_def}.
\end{prop}
\begin{proof}
Let $\big(t_{0},x_{0},\lambda_{0}\big)\in H$ and let $(\tau,x):[t_{0},T]\rightarrow \R^{n}$ be a solution of \eqref{APSP}. 
Then the optimality principle yields, for all $s\in [0, T-t_{0}]$, 
$$ \big(\lambda_{0}\leq)\; T(t_{0},x_{0})\leq s+ T\big(t_{0}+s,x(t_{0}+s)\big),$$
which implies that $\big(\tau(s),x(s),\lambda(s)\big)=\big(t_{0}+s, x(t_{0}+s),\lambda_{0}-s\big)\in H$ for all $s\in [T-t_{0}]$. The proof is concluded.
\end{proof}

The following converse properties are also valid.

\begin{prop} \label{CSI}
Let the set-valued maps $C$ and $G$ satisfy assumptions $(H_{C})$ and $(H_{G})$, respectively, 
and let $S\subseteq \R^n$ be closed and such that \eqref{non-empty} is valid. Take a function $\theta:\mathrm{graph}(C)\rightarrow \R \cup \{\infty\}$ 
such that $\theta(t,x)>0$ for all $(t,x)\in\mathrm{graph}(C)$ for which $x\notin S$ and $\theta(t,x)=0$ for all 
$(t,x)\in\mathrm{graph}(C)$ for which $x\in S$. Then
\begin{itemize}
\item[(i): \;] if $\theta$ is lower semicontinuous and $\mathrm{epi}(\theta)$ is weakly invariant with 
respect to $\Gamma$, defined in \eqref{Gamma_def}, then 
$$\theta(t,x)\geq T(t,x)$$
for all $(t,x)\in \mathrm{graph}(C)$;
\item[(ii): \;] if $\theta$ is upper semicontinuous and $\mathrm{hypo}(\theta)$ is strongly invariant with 
respect to $\Gamma$, then 
\[\theta(t,x)\leq T(t,x)\]
for all $(t,x)\in \mathrm{graph}(C)$.
\end{itemize}
\end{prop}
\begin{proof}
(i): The statement needs a verification just at points $(t,x)\in \mathrm{dom}(\theta)$ with $x\notin S$. 
Let $(\tau, y, \lambda): [0,T)\rightarrow \R^{1+n+1}$ be a solution, defined on a maximal interval of existence, of 
$$ (\dot{\tau}(t), \dot{y}(t), \dot{\lambda}(t))\in \Gamma(\tau(t), x(t), \lambda(t)), \qquad (\tau(0), y(0), \lambda(0))=(t,x,\theta(t,x))$$
such that
\begin{equation}\label{inequality}
\theta(t,x)-s \geq \theta(t+s, y(s))  >0 \qquad \forall \; s\in[0,T).
\end{equation}
By the above property, the end time $T$ must be finite, otherwise we contradict the hypothesis $\theta(t,x)\geq 0$ for all 
$(t,x)\in \mathrm{graph}(C)$. Observe also that, by the boundedness of $\Gamma$, $y$ can be extended up 
to $s=T$. We claim that $\bar{\theta}:= \liminf_{s\uparrow T} \theta(t+s, y(s))= \theta(t+T, y(T))=0 $. 
Indeed, if not we can prolong $(\tau, y, \lambda)$ beyond $T$ still remaining in $E$ and with $y(T+s)\notin S$ for all 
$s >0$ small enough, hence violating the maximality of $T$.
Thus, from our assumptions on $\theta$, it follows that
$y(T)\in S$ and that $T\geq T(t,x)$. Recalling \eqref{inequality}, the verification of the claim (i) is concluded.

(ii): The statement requires verification only at points $(t_{0},x_{0})\in \mathrm{dom}(T)$ with 
$x_{0}\notin S$. Fix $\eta >0$ and let $y_{\eta}$ be a solution of \eqref{PSP} defined on $[t_{0},T]$ with 
$T-t_{0}< T(t_{0},x_{0})+\eta $ and $y_{\eta}(T)\in S$. Define $(\tau(s), y(s), \lambda(s)):= (t_{0}+s, y(t_{0}+s), 
\theta(t_{0}, x_{0})-s),$ $s\in[0, T-t_{0}]$. Observe that $(\tau, y, \lambda)$ is a solution of the Cauchy problem
$$ (\dot{\tau}(t), \dot{y}(t), \dot{\lambda}(t))\in \Gamma(\tau(t), x(t), \lambda(t)), \qquad (\tau(0), y(0), 
\lambda(0))=(t_{0},x_{0},\theta(t_{0},x_{0})).$$
By strong invariance assumption, $\theta(t_{0}+s, y(s))\geq \theta(t_{0},x_{0})-s$ for all $s\in [0, T-t_{0}]$. 
By taking $s=T-t_{0}$, using the relation $y(T-t_{0})=y_{\eta}(T)\in S$, we obtain that 
\[
\theta (t_{0},x_{0}) -(T-t_{0}) \leq 0 
\]
hence
\[
\theta (t_{0},x_{0}) \leq T-t_{0}  < T(t_{0},x_{0})+ \eta.
\]
Since $\eta$ was arbitrarily chosen, the proof is complete.
\end{proof}
In the case in which $T$ is continuous, from the above results one immediately obtains a verification theorem.
\begin{cor}\label{verif}
Let the set-valued maps $C$ and $G$ satisfy assumptions $(H_{C})$ and $(H_{G})$, respectively, and let $S\subseteq \R^n$ be closed
and such that \eqref{non-empty} is valid. Take a continuous function $\theta:\mathrm{graph}(C)\rightarrow \R $ 
such that $$\theta(t,x)>0\qquad \forall\;(t,x)\in\mathrm{graph}(C)\qquad \mathrm{for}\; \mathrm{which}\quad x\notin S,$$
$$\theta(t,x)=0\qquad \forall \; (t,x)\in\mathrm{graph}(C)\qquad \mathrm{for}\; \mathrm{which} \quad x\in S.$$
Assume moreover that $\mathrm{epi}(\theta)$ is weakly invariant and $\mathrm{hypo}(\theta)$ is strongly 
invariant with respect to the dynamics $\Gamma$ defined in \eqref{Gamma_def}.
Then $\theta(t,x)=T(t,x)$ for every $(t,x)\in \mathrm{dom}(\theta)$.
\end{cor}

\section{Hamilton-Jacobi inequalities}\label{H-J_ineq_sec}
The main result of this part, which follows from combining Sections \ref{INV} and \ref{DP}, is the culmination of the whole paper.
Before stating the theorem, we set the Hamiltonian notation. Let $C$ and $G$ be the set-valued maps which have 
come throughout the paper and let $\theta: \mathrm{graph}(C) \rightarrow \R \cup \{\infty\}$ be a function. 
For $(\tau,x,\lambda)\in \mathrm{epi}(\theta)$ and $p\in N^{P}_{\mathrm{epi}(\theta)}(\tau,x,\lambda)$, define
\begin{equation}\label{H-_Gamma}
H_{-}(\tau,x,\lambda,p):= \min_{v\in\{0\}\times \{ -N_{C(\tau)}(x)\cap(L+M)\B\}\times\{0\}} v\boldsymbol{\cdot} p + \min_{v\in \{1\} 
\times G(\tau,x)\times \{ -1\}} v\boldsymbol{\cdot} p,
\end{equation}
while for $(\tau,x,\lambda)\in \mathrm{hypo}(\theta)$ and $p\in N^{P}_{\mathrm{hypo}(\theta)}(\tau,x,\lambda)$,
\begin{equation} \label{H+_Gamma}
H_{+}(\tau,x,\lambda,p):= \min_{v\in\{0\}\times \{ -N_{C(\tau)}(x)\cap(L+M)\B\}\times\{0\}} v\boldsymbol{\cdot} p + \max_{v\in \{1\} 
\times G(\tau,x)\times \{ -1\}} v\boldsymbol{\cdot} p.
\end{equation}

\begin{thm}\label{main}
Let the set-valued maps $C$ and $G$ satisfy assumptions $(H_{C})$ and $(H_{G})$, respectively. 
Let $S\subset \R ^{n}$ be closed and consider the minimum time $T(t_{0},x_{0})$ to reach $S$ subject to \eqref{PSP}. 
Assume that $T$ is continuous on $\mathrm{graph}(C)$. Then $T$ is the unique continuous function satisfying the following properties:
$$T(t,x)>0\qquad \forall\;(t,x)\in\mathrm{graph}(C)\qquad \mathrm{for}\; \mathrm{which}\quad x\notin S,$$
$$T(t,x)=0\qquad \forall \;(t,x)\in\mathrm{graph}(C)\qquad \mathrm{for}\; \mathrm{which} \quad x\in S,$$
\begin{itemize}
\item[$(H_{-})$] $\qquad H_{-}(t,x,T(t,x), p) \leq 0 \qquad \forall\; (t,x)\in\mathrm{graph}(C), 
\quad \forall \; p \in N^{P}_{\mathrm{epi}\,(T)}(t,x,T(t,x)),$
\item[$(H_{+})$] $\qquad H_{+}(t,x,T(t,x), p) \leq 0 \qquad \forall\; (t,x)\in\mathrm{graph}(C), 
\quad \forall \; p \in N^{P}_{\mathrm{hypo}\, (T)}(t,x,T(t,x))$.
\end{itemize}
\end{thm}

\begin{rem}
The inequalities $(H_{-})$ and $(H_{+})$ contain both a boundary condition (at the boundary of $\mathrm{graph}(C)$) 
and proximal inequalities in the interior of $\mathrm{graph}(C)$, if any. More precisely assume, in addition to the 
hypotheses of Theorem \ref{main}, that $\mathrm{int}\,C(t) \neq \emptyset$ for all $t\geq0$. 
Then $(H_{-})$ yields, for all $t>0$, $x\in \mathrm{int}\,C(t)$  and $p\in \partial_{P} T(t,x)$, if any, 
$$\min_{w\in \{1\}\times G(t,x)} w\boldsymbol{\cdot} p \leq 0$$
and $p\in \partial^{P} T(t,x)$, if any, 
$$ - \min_{w\in \{ 1\}\times G(t,x)} w\boldsymbol{\cdot} p = \max_{w\in \{ 1\}\times G(t,x)} w\boldsymbol{\cdot} (-p)\leq 0$$
In view of Rockafellar's horizontality theorem \cite{Ro}, it is well known that is not necessary to test 
the Hamiltonian inequalities at horizontal normal vectors (see \cite{WoZh}, pg. 1059). 
In particular, if both $\partial_{P}T(t,x)$ and $\partial^{P}T(t,x)$ are non-empty, then ($\nabla T(t,x)$ exists and) 
$$ \frac{\partial T}{\partial t}(t,x) + \min_{w\in G(t,x)} w\boldsymbol{\cdot} \frac{\partial T}{\partial x}(t,x)=0.$$
\end{rem}
\begin{proof}
The result follows from combining Proposition \ref{WI} with Theorem \ref{wih} and Proposition \ref{SI} with Theorem 
\ref{sih} and from Corollary \ref{verif}. We just observe that results of Section \ref{INV} can be applied 
to the augmented dynamic $\Gamma$ defined in \eqref{Gamma_def} because we can rewrite
$$ \Gamma(t,x,\lambda)= \{1\} \times \big(-N_{C(\tau)\times \R} (x,\lambda) + G(\tau,x)\times \{ -1\}\big).$$
The proof is concluded.
\end{proof}

We devote the last part of this section to a statement of the problem and of the main result in the autonomous case. 
This is sometimes called a state constrained problem with reflecting boundary (see \cite{Se05}).
The assumptions now are:
\begin{equation}\label{ass_aut}
\begin{split}
&\text{$C\subset \R^{n}$ is closed and $r$-prox-regular,}\\
&\text{$G : \R ^{n} \leadsto \R^{n}$   
is a $L_{C}$-Lipschitz continuous set-valued map with compact and convex values,}\\
&\text{the closed target $S\subset \R^{n}$ is such that $S\cap C \neq \emptyset$.
}
\end{split}
\end{equation}
Consider the Cauchy problem
\begin{equation}\label{PSPa}
\left\{ \begin{array}{l}
\dot{x}(t)\,\in \,  -N_{C}(x(t)) + G(x(t))\quad \mathrm{a.e.}\; t > 0,
\\
x(0)= x_{0}\in C\;
\end{array}\right.
\end{equation}
and define the minimum time to reach $S$ subject to \eqref{PSPa} by setting
\[
T(x_{0})= \inf \{ \alpha \geq 0 : \; x(\alpha)\in S, \quad x\quad \mathrm{is}\quad\mathrm{a}\quad\mathrm{solution}\quad\mathrm{of}\quad\eqref{PSPa} \}
\]
We rephrase our main result for this case:

\begin{thm} \label{main_a} Let $C$, $S$, $G$ satisfy the assumptions \eqref{ass_aut}. 
Suppose that the minimum time function is $T$ to reach $S$ subject to \eqref{PSPa} is continuous. 
Then $T$ is the unique continuous function satisfying the following properties
\begin{itemize}
\item[ ]$$ T(x)=0 \qquad \forall \; x\in S\cap C,$$
\item[ ]$$T(x)>0 \qquad \forall \; x\in S \setminus C,$$
\end{itemize}
$$(Ha_{-})\qquad \min_{v\in (-N_{C}(x)\cap L_{C}B) \times \{ 0\}}+\min_{v\in G(x) \times \{-1\}}v\boldsymbol{\cdot} p 
\leq 0 \qquad \forall\; x\in C, \quad \forall \; p \in N^{P}_{\mathrm{epi}(T)}(x,T(x)),$$
$$(Ha_{+})\qquad \min_{v\in (-N_{C}(x)\cap L_{C}B) \times \{ 0\}}+\max_{v\in G(x) \times \{-1\}}v\boldsymbol{\cdot} p 
\leq 0 \qquad \forall\; x\in C \quad \forall \; p \in N^{P}_{\mathrm{hypo}(T)}(x,T(x)).$$
In particular, at any interior point of $C$ (if any), the above Hamiltonian inequalities become the classical proximal solution conditions
$$\min_{v\in G(x) \times \{-1\}}v\boldsymbol{\cdot} p \leq 0 \qquad \forall\; x\in C, \quad \forall \; p \in \partial_{P}T(x),$$
$$\min_{v\in G(x) \times \{-1\}}v\boldsymbol{\cdot} p \geq 0 \qquad \forall\; x\in C, \quad \forall \; p \in \partial^{P}T(x).$$
\end{thm}
Again, observe that the Hamiltonian inequalities contain some boundary conditions.

\section{On the continuity of the minimum time function}\label{sec:cont}

Throughout the whole paper we have assumed that $T$ is (finite and) continuous on the whole
 graph of the moving set $C$. In this section we give a sufficient condition for this property. 
We will state a controllability result general enough to cover the example in Section \ref{Ex}. 
The sufficient condition we propose is of Petrov type (with variable coefficient), see, e.g., \cite[Sect. 8.2]{CS}. 
Much finer results can be found in the literature (see, e.g. \cite{LeMa}, \cite{Ma}, \cite{MaRi} and references therein) 
and we believe they can be adapted to \eqref{PSP}. 

The result we are going to present follows the lines of Theorem 1 in \cite{Ma}. However, in order to 
overcome the lack of Lipschitz continuity of the normal cone $N_{C(\cdot)}(\cdot)$, we need to impose a 
Petrov-like condition which takes into account also normals at nearby points.

\begin{prop} \label{continuity}
Let the set-valued maps $C$ and $G$ satisfy $(H_{C})$ and $(H_{G})$, respectively. 
Let $S\subset \R^{n}$ be closed and satisfying \eqref{non-empty} and the internal sphere condition. 
Assume furthermore that there exist a continuous non decreasing function $\mu: 
[0,\infty)\rightarrow [0,\infty)$, with $\mu(0)=0$ and
$$ \mu(\rho)>0,\qquad\qquad \int_{0}^{\rho} \frac{dr}{\mu(r)}< \infty\qquad \forall \; \rho>0,$$
and $\delta>0$ with the following property: for all $t\geq0$ and $x\in C(t)$, there exist $\bar{v}\in G(t,x)$ and 
$\bar{\xi} \in \partial^{P}d_{S}(x)$ such that, for all $(s,y)\in \mathrm{graph(C)}$ with 
$\|(s,y)-(t,x)\|\leq \delta$ and all $p\in N_{C(s)}(y)\cap L \B$, one has
$$ (\bar{v}-p)\boldsymbol{\cdot} \bar{\xi} \leq -\mu(d_{S}(x)).$$
Then  $T:\mathrm{graph}(C)\rightarrow \R^{+}$ is (finite and) continuous.
More precisely, for all $(t_{1},x_{1}), (t_{2},x_{2}) \in \mathrm{graph}(C)$, we have
\begin{equation}\label{mod_cont}
|T(t_{2},x_{2})-T(t_{1},x_{1})| \leq \int_{0}^{e^{KT}\|x_{2}-x_{1}\|+K'\sqrt{|t_{1}-t_{2}|}}\frac{2}{\mu(r)}\, dr,
\end{equation}
for suitable constants $K$ and $K'$ depending only on $C$, $G$ and the dimension of the space, 
provided $T(t_{1},x_{1})$ and $T(t_{2},x_{2})$ are both not larger than a given constant.
\end{prop}
\begin{proof}
Fix $(t_{0},x_{0})\in \mathrm{graph}(C)$. We claim first that 
\begin{equation}\label{est0}
T(t_{0},x_{0}) \leq 2 \int_{0}^{d_{S}(x_{0})} \frac{ds}{\mu(s)}
\end{equation}
(hence it is finite). Let $g$ be a $L_G$-Lipschitz selection from $G$ such that $g(t_{0},x_{0})=\bar{v}$, 
and let $y_{0}$ be the solution of the Cauchy problem
\[ 
\:\left\{ \begin{array}{l}
\dot{x}(t)\,\in \,  -N_{C(t)}(x(t)) + g(t,x(t))\quad \mathrm{a.e.}\; t > t_{0},
\\
x(t_{0})=x_{0}\in C(t_{0})\;,
\\
x(t)\in C(t)\quad \mathrm{for}\;\mathrm{all} \quad t > t_{0}.
\end{array}\right.
\ \\
\]
Fix $t>t_{0}$ and observe that, by the uniform semiconcavity of $d_{S}$ (recall \eqref{semiconc}), for a suitable $p(s)\in N_{C(s)}(y(s))$,
\begin{equation}
\begin{split}
d_{S}(y(t)) & \leq d_{S}(x_{0}) + \int^{t}_{t_{0}} \bar{\xi}\boldsymbol{\cdot}(-p(s)+g(s,y(s)))ds +\frac{1}{r}\|y(t)-x_{0}\|^{2} 
\\
& \leq  d_{S}(x_{0}) + \int^{t}_{t_{0}} \bar{\xi}\boldsymbol{\cdot}(-p(s)+g(t_{0},x_{0}))ds
\\
& \qquad + \int^{t}_{t_{0}} L_G (L+1)(s-t_{0})ds + \frac{L^{2}}{r}(s-t_{0})^{2}
\\
& \leq d_{S}(x_{0}) - (t-t_{0})\mu(d_{S}(x_{0})) + K|t-t_{0}|^{2},
\end{split}
\end{equation}
where $L$ was defined in \eqref{defL} and $K=\frac{L_G(L+1)}{2}+\frac{L^{2}}{r}$,
and we have assumed that $|t-t_{0}| \leq \frac{\delta}{L+1}$.
With the further assumption that
$$|t-t_{0}| \leq \frac{\mu(d_{S}(x_{0}))}{2K},$$
we obtain the estimate
$$ d_{S}(y(t)) \leq d_{S}(x_{0}) - \frac{\mu(d_{S}(x_{0}))}{2}(t-t_{0}).$$
Set 
$$h_{1} = \min \Big\{ \frac{\delta}{L+1}, \frac{\mu(d_{S}(x_{0})}{2K} \Big\} .$$
By repeating the same argument inductively, we construct a sequence of points $\{ y_{i}\}$, of numbers $h_{i}>0$, and of solutions $y_{i}$ of
\[ 
\:\left\{ \begin{array}{l}
\dot{x}(t)\,\in \,  -N_{C(t)}(x(t)) + G(t,x(t))\quad \mathrm{a.e.}\; t > t_{i},
\\
x(t_{i})=y_{i}\in C(t_{i})\;,
\\
x(t)\in C(t)\quad \mathrm{for}\;\mathrm{all} \quad t > t_{i},
\end{array}\right.
\ \\
\]
where $t_{i}=t_{0}+\sum_{j=1}^{i}h_{j}$, satisfying the properties, for $i=0, 1,$ ....,
\begin{equation} \label{est1}
\begin{split}
y_{i+1} & = y_{i}(t_{i+1})
\\
d_{S}(y_{i+1}) & \leq d_{S}(y_{i}) \frac{\mu(d_{S}(y_{i}))}{2}h_{i+1}
\\
h_{i+1} & = \min \Big\{\frac{\delta}{L+1}, \frac{\mu(d_{S}(y_{i}))}{2K}\Big\}.
\end{split}
\end{equation}
Since the sequence $\{ d_{S}(y_{i})\}$ is decreasing and $\mu(\rho)>0$ if $\rho>0$, we obtain that 
$\lim_{i\rightarrow \infty} d_{S}(y_{i})=0$. We wish now to estimate the sum $\sum_{i=1}^{\infty}h_i$, 
namely the time needed by the trajectory of \eqref{PSP} obtained by patching together $y_{i}$ to reach $S$. Indeed \eqref{est1} yields
\begin{equation} \label{est2}
\sum_{i=1}^{\infty} h_{i} \leq 2 \sum_{i=1}^{\infty}\frac{d_{S}(y_{i})-d_{S}(y_{i+1})}{\mu(d_{S}(y_{i}))}.
\end{equation}
Since the right hand side of \eqref{est2} is a Riemann sum of the finite integral $\int_{0}^{d_{S}(x_{0})} \frac{ds}{\mu(\rho)}$, 
we have proved that the constructed trajectory of \eqref{PSP} reaches $S$ in finite time. Since $\mu$ is non decreasing, 
the right hand side of \eqref{est2} is less or equal to $2\int_{0}^{d_{S}(x_{0})}\frac{ds}{\mu(s)}$ and so the claim is verified.

We will deduce now from \eqref{est0} that $T$ is continuous and satisfies \eqref{mod_cont}. To this aim, we use a 
parametrisation theorem (see Theorem 9.7.2 in \cite{AuFr}). According to this result, there exists a single-valued map 
$g:\mathrm{graph}(C)\times \B\rightarrow \R^{n}$ with the following properties:
\begin{itemize}
\item[(a)\;] for all $(t,x)\in\mathrm{graph}(C),$ $G(t,x)=g(t,x,\B)$;
\item[(b)\;] for all $u\in \B,$ $g(\cdot,\cdot, u)$ is Lipschitz continuous with constant $cL_{G}$, where $c$ is a suitable number independent of $G$;
\item[(c)\;] for all $(t,x)\in \mathrm{graph}(C)$, $g(t,x,\cdot)$ is Lipschitz continuous with constant $c$. 
\end{itemize}
Fix $(t_{1},x_{1})$ and $(t_{2},x_{2})$ and consider the Cauchy problems for augmented dynamics
\begin{equation}\label{PASP} 
\left\{ \begin{array}{l}
\dot{\tau}(t) = 1
\\
\dot{x}(t)\,\in \,  -N_{C(\tau(t))}(x(t)) + g(\tau(t), x(t),u)\quad \mathrm{a.e.}\; t > 0,
\\
(\tau(0), x(0))= (t_{i}, x_{i}), \quad i=1,2,
\end{array}\right.
\end{equation}
Let $u_{1}$ be a control taking values in $\B$ steering $(t_{1},x_{1})$ to $\R \times S$ in the optimal 
time $T(t_{1},x_{1})$ and let $(\tau_{i},y_{i})$ be the solution of \eqref{PASP} with $u_1$ in place of $u$ for $i=1,2$, respectively.  
If $y_{2}(T(t_{1},x_{1}))\in S$, then obviously $T(t_{2},x_{2})\leq T(t_{1},x_{1}) $. 
Otherwise, set $\bar{x}_{i}=y_{i}(T(t_{1},x_{1}))$, $i=1,2$. Then, recalling \eqref{est0},
\begin{equation} \label{est3}
T(T(t_{1},x_{1}), \bar{x}_{2})\leq 2\int^{d_{S}(\bar{x}_{2})}_{0} \frac{dr}{\mu(r)} 
\leq 2\int_{0}^{\|\bar{x}_{2}-\bar{x}_{1}\|}\frac{dr}{\mu(r)}.
\end{equation}
To conclude the proof, we need to estimate $\|\bar{x}_{1} -\bar{x}_{2}\|$. 
Indeed, we have for all $t\in[0,T(t_{1},x_{1})]$, recalling the $r$-prox-regularity of $C(t)$ (see \eqref{ineqphi}),
\begin{equation}
\begin{split}
(\dot{y}_{1}(t)-g(\tau_{1}(t), y_{1}(t), u_{1}(t)))\boldsymbol{\cdot} (y_{2}(t)-y_{1}(t))  &\geq - \frac{1}{2r}\|y_{2}(t)-y_{1}(t)\|^{2},
\\
(\dot{y}_{2}(t)-g(\tau_{2}(t), y_{2}(t), u_{1}(t)))\boldsymbol{\cdot} (y_{1}(t)-y_{2}(t))  &\geq - \frac{1}{2r}\|y_{2}(t)-y_{1}(t)\|^{2}.
\end{split}
\end{equation}
By summing the above inequalities and using the Lipschitz continuity of $g$,  we obtain
\begin{equation}
(\dot{y}_{2}(t)-\dot{y}_{1}(t))\boldsymbol{\cdot}(y_{2}(t)-y_{1}(t)) \leq \frac{1}{r}\|y_{2}(t)-y_{1}(t)\|^{2}+
    cL_{G}\big(|t_{2}-t_{1}|+\|y_{2}(t)-y_{1}(t)\|\big)\|y_{2}(t)-y_{1}(t)\|,
\end{equation}
namely
\begin{equation}
\frac{d}{dt}\|y_{2}(t)-y_{1}(t)\|^{2} \leq \Big(cL_{G}+\frac{1}{r}\Big)\|y_{2}(t)-y_{1}(t)\|^{2}+ cA|t_{2}-t_{1}|
\end{equation}
for a suitably large constant  $A>0$, depending on $t$. Therefore, for all $t\in [0,T(t_{1},x_{1})],$
\begin{equation}\label{est4}
\|y_{2}(t)-y_{1}(t)\|\leq \sqrt{2}\, e^{\big(cL_{G}+\frac{1}{r}\big)t}\|x_{2}-x_{1}\| + A'\sqrt{|t_{2}-t_{1}|},
\end{equation}
where $A'$ is a suitable constant depending only on $cA$.
Combining \eqref{est3} and \eqref{est4} and exchanging the role of $(t_{1},x_{1})$ and $(t_{2},x_{2}),$ 
we get \eqref{mod_cont}. The proof is concluded.
\end{proof} 

\begin{rem}
Observe that, if $\mu(r)=\sqrt{r}$, then $T(t,x)$ is H\"older continuous with exponent $\frac{1}{2}$ 
with respect to $x$ and $\frac{1}{4}$ with respect to $t$, while if $\mu$ is constant 
(as in the original Petrov condition), then $T$ is Lipschitz continuous with respect to $x$ 
and $\frac{1}{2}$-H\"older continuous with respect to $t$. 
\end{rem}

\section{Examples}\label{Ex}

This section is devoted to two examples in which a candidate minimum time function  is proved to 
be the exact one by verification. Indeed we show by direct inspection that our guessed function 
satisfies the Hamilton-Jacobi inequalities and so, by the sufficiency part of the characterisation, 
it is actually $T$. We plan to use the necessity part in a forthcoming paper, for numerical approximations. 

\subsection{Example 1}

Let the space dimension be $1$ and set, for $t\geq0$,
$$ C(t)= \{ x\in \R : \; -1+t \leq x \leq 2  \}.$$ 
(observe that $C(t)\neq \emptyset$ if and only if $t\leq 3$) and
$$S= \{ x\in \R : \; x\geq 2 \} .$$
Consider the controlled dynamics
$$\dot{x}(t) \in -N_{C(t)}(x(t)) + x(t)+u(t), \qquad |u(t)|\leq1.$$
Due to the non autonomous character of the problem, the minimum time function to reach $S$ is defined through the augmented dynamics
\begin{equation}\label{(E1)}
\left\{ \begin{array}{l}
\dot{\tau}(t) = 1
\\
\dot{x}(t)\,\in \,  -N_{C(\tau(t))}(x(t)) + x(t) + [-1,1]\quad \mathrm{a.e.}\; t > 0,
\\
(\tau(0), x(0))= (t_{0}, x_{0})\;.
\end{array}\right.
\ \\
\end{equation}
The candidate minimum time function $T$ is the following (we omit the computations which lead us to formulate this guess):
\[
T(t,x):=\left\{ \begin{array}{lll}
1+ \log 3 -t &  & \mathrm{for}\quad -1\leq -1+t \leq x \leq -1+e^{t-1}, \quad 0\le t\le 1,\\
\log3 - \log(1+x) &  & \mathrm{for} \quad -1+e^{t-1} < x\leq 2,\quad 0\le t\le 3.
\end{array}\right.
\]
Figure \ref{Fig1} shows the graph of $T$, together with some curves at which the computation of normal cones is performed in the sequel.
Observe that $T$ is continuous in the whole of $\mathcal{C}:=\mathrm{graph}(C),$ as it is 
expected to be, since the assumptions of the controllability result contained in Proposition 
\ref{continuity} are satisfied. Moreover, $T$ vanishes on $\mathcal{C}\cap (\R \times S)$ 
and it is positive on the remainder of $\mathcal{C}$. Therefore, in order to prove that $T$ 
is actually the minimum time function we are looking for, it is enough to verify the Hamilton-Jacobi 
inequalities $(H_{-})$ and $(H_{+})$ stated in Theorem \ref{main}. In this case, they read as follows, for all $(t,x)\in \mathcal{C}:$
\begin{equation} \label{Hm}
\min_{v \in -N_{C(t)}(x),\; |v|\leq4} (0,v,0)\boldsymbol{\cdot}(p_{t},p_{x},p_{z}) + \min_{|w|\leq 1} (1,x+w, -1)\boldsymbol{\cdot} (p_{t},p_{x},p_{z}) \leq 0
\end{equation}
for all  $p=(p_{t},p_{x}, p_{z})\in N_{\mathrm{epi}(T)} (t,x,T(t,x)),$ and
\begin{equation} \label{Hp}
\min_{v \in -N_{C(t)}(x),\; |v|\leq4} (0,v,0)\boldsymbol{\cdot}(p_{t},p_{x},p_{z}) + \max_{|w|\leq 1} (1,x+w, -1)\boldsymbol{\cdot} (p_{t},p_{x},p_{z}) \leq 0
\end{equation}
for all  $p=(p_{t},p_{x}, p_{z})\in N_{\mathrm{hypo}(T)} (t,x,T(t,x))$.

To this aim, observe first that, for $(t,x)$ in the interior of $\mathcal{C}$ (except at the curve 
$\gamma$: $x=-1 +e^{t-1}$, $0\leq t \leq 1$) and $p=(\frac{\partial T}{\partial t}, \frac{\partial T}{\partial x}, \pm 1)$, 
both inequalities are satisfied as an equality. Second, we check \eqref{Hm} and \eqref{Hp} at points $(t,x,p)$ with 
$(t,x)\in \rm{bdry}\, \mathcal{C}$. Set $E=\mathrm{epi}(T)$ and $H=\mathrm{hypo}(T)$ and observe that for a 
horizontal $p$, i. e., $p=(p_{t},p_{x}, 0),$ we need to verify only \eqref{Hp}, since the horizontal normal vectors 
are in common between normal cones of the epigraph and of the hypograph.

The verification procedure is employed as follows. Consider, for instance, the point $(0,-1, 1+\log 3)$. 
Then the normal cones to the epigraph and to the hypograph respectively, together with $N_{C(0)}(-1)$, admit the following representation: 
\begin{equation}
\begin{split}
N_{E}(0,-1, 1+ \log 3)&= \mathrm{cone} \{ (0,-1,0), (1,-1, 0), (-1,0,-1)\},
\\
N_{H}(0,-1, 1+ \log 3)& = \mathrm{cone} \{ (0,-1,0), (1,-1, 0), (1,0,1)\},
\\
N_{C(0)}(-1) &= (-\infty, 0].
\end{split}
\end{equation}
For $p=(0,-1,0)$, \eqref{Hp} reads as 
$$ \min_{v \in 0\leq v\leq4} (-v) + \max_{|w|\leq 1} (1-w)=-4+2=-2 < 0, $$
while for $p=(1,-1,0)$, \eqref{Hp} reads as 
$$ \min_{v \in 0\leq v\leq4} (-v) + \max_{|w|\leq 1} (2-w)=-4+3=-1 < 0. $$
For $p=(-1,0,-1)$, \eqref{Hm} reads as 
$$ \min_{v \in 0\leq v\leq4} (-v) + (-1)+1=-4 < 0, $$
and a similar computation holds for \eqref{Hp} at $(1,0,1)$.

The same procedure can be repeated for the points $(3,2,0)$, $(0,2,0)$ and the segment $\{ (t,2,0): \; 0<t<3 \}.$
In this case, one can easily verify that normal cones of interest in such points can be written as
\begin{equation}
\left\{
\begin{aligned}
N_{E}(3,2,0)&= \mathrm{cone} \{ (0,1,0), (1,-1, 0), (0,-1/3,-1)\},
\\
N_{H}(3,2,0)& = \mathrm{cone} \{ (0,1,0), (1,-1, 0), (0,1/3,1)\},
\\
N_{C(3)}(2)& = \R,
\end{aligned}
\right.
\end{equation}
\begin{equation}
\left\{
\begin{aligned}
N_{E}(0,2,0)&= \mathrm{cone} \{ (0,1,0), (-1,0, 0), (0,-1/3,-1)\},
\\
N_{H}(0,2,0)& = \mathrm{cone} \{ (0,1,0), (-1,0, 0), (0,1/3,1)\},
\\
N_{C(0)}(2)& = [0,\infty),
\end{aligned}
\right.
\end{equation}
\begin{equation}
\left\{
\begin{aligned}
N_{E}(t,2,0)&= \mathrm{cone} \{ (0,1,0), (0,-1/3,-1)\},
\\
N_{H}(t,2,0)& = \mathrm{cone} \{ (0,1,0), (0,1/3,1)\},
\\
N_{C(t)}(2)& = [0,\infty),
\end{aligned}
\right.
\end{equation}
respectively.

It is now a straightforward matter to check the inequalities \eqref{Hp}, \eqref{Hm} evaluated at  points $(3,2,0)$, $(0,2,0)$ and  
along the segment $\{ (t,2,0): \; 0<t<3 \}$, taking into account the $p$'s that generate the normal cones of the epigraphs 
and of the hypographs. (Again, we observe that for horizontal normal vectors it is enough to check just \eqref{Hp}). 

In the same spirit, we consider now the curves contained in $\mathrm{graph}(T):$
\begin{equation}
\begin{split}
& \Gamma_{1} := \{(t,-1+t, 1 +\log3 -t): \; 0 < t < 1 \},
\\
& \Gamma_{2} := \{(0,x, 1 +\log3 ): \;-1 < x < -1 +e^{-1} \},
\\
& \Gamma_{3} := \{(t,-1+t, 1 +\log3 -\log t): \; 1 < t < 3 \},
\\
& \Gamma_{4} := \{(0,x, \log3 - \log(1+x) ): \; -1+e^{-1} < x < 2 \},
\\
& \Gamma_{5}:=\{(t,e^{t-1}-1, 1 +\log3 - t): \; 0 \leq t \leq 1 \}.
\end{split}
\end{equation}
which are either singular for the minimum time function $T$, or lying at the boundary of epi/hypo\-graph of $T$.

An analysis  along $\Gamma_{1}$ provides the representations
\begin{equation}
\begin{split}
N_{E}(t,-1+t,1+ \log 3 -t)&= \mathrm{cone} \{ (1,-1,0), (-1,0, -1) \},
\\
N_{H}(t,-1+t,1+ \log 3 -t) &= \mathrm{cone} \{ (1,-1,0), (1,0, 1)\},
\\
N_{C(t)}(-1+t) &= (-\infty, 0].
\end{split}
\end{equation}
For $p=(1,-1,0),$ \eqref{Hp} reads as
$$ \min_{v \in 0\leq v\leq4} (-v) + \max_{|w|\leq 1} (-(-1+t+w)+1)=-4+3-t=-1-t < 0; $$
for $p=(-1,0,-1)$,   $(H_{-})$ is equal to $0$, as well as $(H_{+})$ for $p=(1,0,1)$.

The cases which comprise $\Gamma_{2}$, $\Gamma_{3}$ and $\Gamma_{4}$ are treated in a similar fashion. 
We mention that the cones related to epigraphs and hypographs for such curves are
\begin{equation}
\left\{
\begin{aligned}
N_{E}(0,x,1+ \log 3)&= \mathrm{cone} \{ (-1,0,0), (-1,0, -1) \},
\\
N_{H}(0,x,1+ \log 3)& = \mathrm{cone} \{ (-1,0,0), (1,0, 1)\},
\\
N_{C(0)}(x) &= \{0\}.
\end{aligned}
\right.
\end{equation}
\begin{equation}
\left\{
\begin{aligned}
N_{E}(t,-1+t,1+ \log 3 -\log t)&= \mathrm{cone} \{ (1,-1,0), (0,-\frac{1}{t}, -1) \},
\\
N_{H}(t,-1+t,1+ \log 3 -\log t) &= \mathrm{cone} \{ (1,-1,0), (0,\frac{1}{t}, 1)\},
\\
N_{C(t)}(-1+t) &= (-\infty, 0].
\end{aligned}
\right.
\end{equation}
\begin{equation}
\left\{
\begin{aligned}
N_{E}(0,x,\log 3 -\log (1+x))&= \mathrm{cone} \{ (-1,0,0), (0,-\frac{1}{x+1}, -1) \},
\\
N_{H}(0,x,\log 3 -\log (1+x)) &= \mathrm{cone} \{ (-1,0,0), (0,\frac{1}{x+1}, 1) \},
\\
N_{C(0)}(x) &= \{ 0 \}.
\end{aligned}
\right.
\end{equation}
for $\Gamma_{2}$, $\Gamma_{3}$ and $\Gamma_{4}$, respectively. For what concerns $\Gamma_{5}$, it turns out 
that $N_{E}=\{ 0\}$ since $\mathrm{graph}(T)$ has an upward kink. The only inequality to check
is therefore \eqref{Hp}, which is confirmed by passing to the limit along the graph of $T$. Finally, at the vertical spots of $\mathrm{epi}(T)$, 
resp. of $\mathrm{hypo}(T)$, the verification procedure is a special case of the computation above.

This shows that $T$ is the minimum time function for the optimal control problem governed by \eqref{(E1)}.

Observe, finally, that at points $(t,-1+t),$ $0\leq t < 1$,  due to the dragging of the moving constraint,
the velocity of the optimal trajectory  is $(1,1)$, 
which is neither the projection of the optimal unconstrained velocity $(1,-1+t+1)=(1,t)$ onto $T_{\mathrm{graph}(C)}(t,-1+t),$ 
nor it is the Cartesian product between $\{ 1\}$ and the projection of $-1+t+1=t$ onto $T_{C(t)}(-1+t).$

\subsection{Example 2}
This example is concerned with a \textit{fixed} set $C$. Let
\[
C =\{ X=(x,y)\in\mathbb{R}^2 : -5\le x \le 5, \: 0\le y \le 4, \: x^2 + (y-2)^2 \ge 1\}.
\]
Observe that $C$ is $r$--prox-regular, with $r=1$.

Let 
\[
G = \{(x,y) : |x| \le 1,\: |x|\le y \le 1 \},
\qquad
F(x) = -N_C(x) + G,
\]
and
\[
S = \{ (x,y) : y\ge 4\}.
\]
Consider the problem of  reaching $S$ in minimum time subject to the dynamics
\[
\dot{X} \in -N_C(X) + G,\; X(0)=X\in C
\]
and let $T(X)$ be the minimum time function. Observe that the assumptions of Proposition \ref{continuity} are satisfied, 
so $T$ is continuous. For future use, let also $T_0(X)$ be the minimum time to reach $S$ subject to the dynamics
\[
\dot{X} \in -N_{C_1}(X) + G,\; X(0)=X\in C, \text{where}
\]
\[
C_1:= \{ (x,y) : -5\le x \le 5, \: 0\le y \le 4\}.
\]
Of course $T$ differs from $T_0$ only in the region
\[
D := \Big\{ (x,y)\in C : -\frac{1}{\sqrt{2}} < x < \frac{1}{\sqrt{2}},\: |x| + 2 - \sqrt{2} < y < 2 -\frac{\sqrt{2}}{2} \Big\}, 
\]
and in $C\setminus D$ we have $T(x,y) = 4-y$ and the optimal trajectory is a segment. 
In the region $D$, we make an educated guess on the the optimal trajectory and then proceed by verification. By symmetry,
we consider only the part $D^+$ of $D$ which is contained in the half plane $\{ (x,y) : x \ge 0\}$. The candidate optimal control is the constant $(1,1)$
until the corresponding trajectory hits the circle $\Delta:=\{ x^2 + (y-2)^2 =1\}$ and then the 
reaction of the constraint switches on by projecting $(1,1)$ onto the
tangent line. The candidate optimal trajectory then moves along the circle until the point $P:=\big(\frac{\sqrt{2}}{2},2-\frac{\sqrt{2}}{2}\big)$
-- having tangent parallel to $(1,1)$ -- is reached, and then proceeds again according to the control $(1,1)$. 
More precisely, fix $2-\sqrt{2} < y_0 \le 1$.
Then the minimum time to reach $\Delta$ from $(0,y_0)$ with optimal control $(1,1)$ is
\begin{equation}\label{T1}
T_1(y_0) := \frac{1}{2} \left(-\sqrt{-y_0^2+4 y_0-2}-y_0+2\right)
\end{equation}
and the corresponding point on $\Delta$ is 
$X(y_0)= \left(T_1(y_0),\frac{1}{2} \left(-\sqrt{-y_0^2+4 y_0-2}+y_0+2\right)\right)$. 
The angular velocity $\alpha$ of the candidate optimal trajectory satisfies the ODE
\[
\dot{\alpha} = \sqrt{2} \cos \left(\frac{\pi}{4}-\alpha\right)
\]
and the initial condition is $\alpha (0) = \arcsin T_1(y_0)$. The general solution of the ODE is
\[
\alpha (t) = 2\arctan\Big(\sqrt{2} \tanh \big(\sqrt{2} t +c\big) +1\Big),
\]
where $0\le c\le \pi/4$ is a constant. The corresponding time to reach the angle $\pi/4$ starting from $\alpha (0)$ is
\[
T_2(y_0) :=  \sqrt{2}\Bigg[\tanh ^{-1}(1-\sqrt{2}\big) -\tanh ^{-1}\Bigg(\frac{\tan \Big(\frac{1}{2} \arcsin T_1(y_0)\Big)-1}{\sqrt{2}}\Bigg)\Bigg].
\]
Therefore, the candidate minimal time to reach $S$ from $(0,y_0)$ is the sum of the time $T_1$ needed to hit $\Delta$, of the time $T_2$
needed to slide along $\Delta$ until $P$ is reached, and of the time needed to reach $S$ from $P$, namely it is
\[
T_3(y_0) := T_1 (y_0) + T_2 (y_0) + 2 +\sqrt{2}/2.
\] 
In general, for $(x_0,y_0)\in D^+$, we have
\[
T(x_0,y_0) = T_3 (y_0-x_0) - x_0, 
\]
whence, if $(x,y)$ belongs to the interior of $D^+$, 
\[
\begin{split}
\partial_x T (x, y ) &= -1 - T_1'(y-x)\Bigg( 1 - \frac{1}{\sqrt{1-T_1^2(y-x)}\big(T_1(y-x) + \sqrt{1-T_1^2 (y-x)}\big)}\Bigg)\\
\partial_y T (x, y ) &= T_1'(y-x)\Bigg( 1 - \frac{1}{\sqrt{1-T_1^2(y-x)}\big(T_1(y-x) + \sqrt{1-T_1^2 (y-x)}\big)}\Bigg),
\end{split}
\]
where we recall from \eqref{T1} that
\[
T_1'(\xi)=  -\frac{1}{2} +\frac{\xi-2}{2\sqrt{-\xi^2+4 \xi-2}}. 
\]
For future use, we compute (the limit of) the derivative of $T$ in the direction (normal to $\Delta$) $v=(x,-\sqrt{1-x^2})$ at points of the arc
$(x, 2- \sqrt{1-x^2})$, $0<x<\frac{1}{\sqrt{2}}$. To this aim, observe first that, for $0<x<\frac{1}{\sqrt{2}}$, we have
\[
T_1\big( 2-\sqrt{1-x^2}-x\big)-x=  \frac{1}{2} \left(\sqrt{1-x^2}-\sqrt{1-2 x \sqrt{1-x^2}}-x\right)=0
\]
and, taking into account the above facts
\[
T_1'\big( 2-\sqrt{1-x^2}-x\big) = -\frac{\sqrt{1-x^2}}{\sqrt{1-2x\sqrt{1-x^2}}}.
\]
Then we receive
\[
\begin{split}
\partial_x T \big( x, 2-\sqrt{1-x^2}\big) & = -1 + \frac{\sqrt{1-x^2}}{\sqrt{1-2x\sqrt{1-x^2}}}
                     \left( 1- \frac{1}{\sqrt{1-x^2}\left( x + \sqrt{1-x^2}\right) }\right)\\
\partial_y T \big( x, 2-\sqrt{1-x^2}\big)  &= - \frac{\sqrt{1-x^2}}{\sqrt{1-2x\sqrt{1-x^2}}}
                     \left( 1- \frac{1}{\sqrt{1-x^2}\left( x + \sqrt{1-x^2}\right)} \right).
\end{split}
\]
By simplifying, using the identity
\[
-x+\sqrt{1-x^2}=\sqrt{1-2x\sqrt{1-x^2}},\qquad |x|\le \frac{\sqrt{2}}{2},
\]
we obtain that the normal derivative $\partial T/\partial v$ at $(x, 2- \sqrt{1-x^2})$, $0<x<\frac{1}{\sqrt{2}}$,
namely the scalar product
\begin{equation}\label{normder}
\left(\partial_x T ( x, 2-\sqrt{1-x^2}), \partial_y T ( x, 2-\sqrt{1-x^2})\right)\boldsymbol{\cdot} \left( x,-\sqrt{1-x^2}  \right)=0.
\end{equation}
Observe furthermore that $T$ is not Lipschitz around the point $\big(\frac{\sqrt{2}}{2},2-\frac{\sqrt{2}}{2}\big)$, since at 
$\xi = 2-\frac{\sqrt{2}}{2}-\frac{\sqrt{2}}{2} = 2 -\sqrt{2}$ we have $T_1'(\xi) = -\infty$.

Now we wish to verify that $T$ is indeed the minimum time by proving that it satisfies both Hamiltonian conditions. 
Set $E$ to be the epigraph of $T$ and $H$ to be its hypograph.
We wish to verify that
\begin{align}
\mathcal{H}_-(X,p) &:= \min_{v\in \big((-N_{C}(X))\times \{ 0\}\big)\cap \sqrt{2}\mathbb{B}} v\boldsymbol{\cdot} p +
\min_{v\in G} (v,-1)\boldsymbol{\cdot} p\le 0,
\quad \forall p\in N^P_{E}(X,T(X))\label{cond-}\\
\mathcal{H}_+(X,p) &:= \min_{v\in \big((-N_{C}(X))\times \{ 0\}\big)\cap \sqrt{2}\mathbb{B}} v\boldsymbol{\cdot} p +
\max_{v\in G} (v,-1)\boldsymbol{\cdot} p\le 0,
\quad \forall p\in N^P_{H}(X,T(X)).\label{cond+}
\end{align}
We exploit the symmetry and proceed only in the region $C^+:= \{ (x,y)\in C : x \ge 0\}$. 
To this aim, observe first that both in the interior and in the exterior of the region
$D^+$ (within $C$) the function $T$ is differentiable and its gradient satisfies both Hamiltonian conditions as an equality. 
In fact, since in both conditions the first summand vanishes, we have first to verify that the (minimized)
Hamiltonian equals $-1$ at differentiability points. In int$D^+$, the minimum is attained by choosing $v = (1,1)$
and it is evident that in this case the condition is verified (i.e., $\partial_x T (x, y ) + \partial_y T (x, y ) \equiv -1$), 
while in the intersection of $C$ with the exterior of $D^+$ the minimum is attained at any point of $G$ of the type
$(x,1)$, $x\ge 0$, and in this case $\partial_y T (x, y ) \equiv -1$. 
Thus also at the part of the boundary of $D^+$ which is contained
in the interior of $C$ the condition is verified, since it is easy to check by direct inspection that
the normal cone to the epigraph (resp. hypograph) of $T$ is either zero, 
or is a convex combination of limiting normal vectors. Therefore, we are left to check conditions 
\eqref{cond-} and \eqref{cond+} at points of the graph corresponding to the circle 
$\Delta$, at the two relevant vertices of the rectangle, at the three relevant
segments of its perimeter (without endpoints), and at the cylinder as well as at the vertical 
rays/faces parallel to the $t$-axis placed above (resp.~below) 
the above mentioned points/sets.

Let us consider first the point $(5,0,4)$. Then
\[
\begin{split}
N_E(5,0,4) &= \mathrm{cone}\, \{ (1,0,0), (0,-1,0), (0,-1,-1)  \}\\
N_H(5,0,4) &= \mathrm{cone}\, \{ (1,0,0), (0,-1,0), (0,1,1) \}\\
-N_C (5,4) &= \mathrm{cone}\, \{ (-1,0),(0,1) \}.
\end{split}
\]
Of course, it is enough to check \eqref{cond-}, resp.~\eqref{cond+},
at the vectors generating $N_E$, resp.~$N_H$, and $N_C$. If $p=(1,0,0)$, then $\mathcal{H}_-\le\mathcal{H}_+=-\sqrt{2}$. 
If $p=(0,-1,0)$, then $\mathcal{H}_+ = -\sqrt{2}$. If $p=(0,-1,-1)$, then $\mathcal{H}_- = -\sqrt{2}$. 
If $p=(0,1,1)$, then $\mathcal{H}_+ = 0$.

Let us consider now the points $(5,y,4-y)$, with $0<y<4$, which belong to graph$(T)$. Then
\[
\begin{split}
N_E(5,y,4-y) &= \mathrm{cone}\, \{ (1,0,0), (0,-1,-1)  \}\\
N_H(5,y,4-y) &= \mathrm{cone}\, \{ (1,0,0), (0,1,1) \}\\
-N_C (5,y) &= \mathrm{cone}\, \{ (-1,0) \}.
\end{split}
\]
If $p=(1,0,0)$, then $\mathcal{H}_+ = 1-\sqrt{2}$. If $p=(0,-1,-1)$, then $\mathcal{H}_- = 0$, 
and if $p=(0,1,1)$, then $\mathcal{H}_+ = 0$ as well.

Let us consider now the point $(5,4,0)$. Then
\[
\begin{split}
N_E(5,4,0) &= \mathrm{cone}\, \{ (1,0,0), (0,1,0), (0,-1,-1)  \}\\
N_H(5,4,0) &= \mathrm{cone}\, \{ (1,0,0), (0,1,0), (0,1,1) \}\\
-N_C (5,4) &= \mathrm{cone}\, \{ (-1,0),(0,-1) \}.
\end{split}
\]
If $p=(1,0,0)$, then $\mathcal{H}_+ = 1-\sqrt{2}$. 
If $p=(0,1,0)$, then $\mathcal{H}_+ = 1-\sqrt{2}$. If $p=(0,-1,-1)$, then $\mathcal{H}_- = -1$. 
If $p=(0,1,1)$, then $\mathcal{H}_+ = 0$.

We pass now to considering the points of the two arcs $A_1:=\big\{ (x,y,T(x,y)): (x,y) \in \Delta, x\ge 0, 1 < y < 2-\frac{\sqrt{2}}{2}\big\} $ and
$A_2:=\big\{ (x,y,T(x,y)): (x,y) \in \Delta,  x\ge 0, 2-\frac{\sqrt{2}}{2} < y < 3 \big\}$. In both cases we have
\[
\begin{split}
N_E(x,y,T(x,y)) &=\mathrm{cone}\, \{ (-x,2-y,0), (\partial_xT(x,y),\partial_yT(x,y),-1) \}\\
N_H(x,y,T(x,y)) &=\mathrm{cone}\, \{ (-x,2-y,0), (-\partial_xT(x,y),-\partial_yT(x,y),1) \}\\
-N_C(x,y)       &=\mathrm{cone}\, \{ (x,y-2) \}.
\end{split}
\]
Along the arc $A_1$, if $p= (-x,2-y,0)$, we have that $\mathcal{H}_- = -2$ and $\mathcal{H}_+ \le 0$, 
while if $p= (\partial_xT(x,y),\partial_yT(x,y),-1)$, taking into account
that $\nabla T$ on $A_1$ has vanishing scalar product with $(x,y-2)$ and also that $\partial_x T + \partial_y T \equiv -1$, we have that
$\mathcal{H}_- = 0$. If $p=(-\partial_xT(x,y),-\partial_yT(x,y),1)$, by the same argument we obtain $\mathcal{H}_+ = 0$. 
Along the arc $A_2$, we obtain $\mathcal{H}_-=\mathcal{H}_+\le 0$, where in the first summand one must choose $v=0$ if $y < 2$.

The verification of \eqref{cond-} and \eqref{cond+} at junction points $(0,y)$, $0\le y \le -1$, where $T$ is nonsmooth, is done by passing to the limit
from the interior of $D^+$ and exploiting the symmetry, while the verification of \eqref{cond-} and \eqref{cond+} at the 
remaining points (laying on vertical spots) is easier and its analysis is contained in the previous cases.

Thus, $T$ is the minimum time function.

\section{Conclusions and comparison with the H-J theory for state constrained optimal control problems}\label{sec:comp}

Our main result is a characterisation of the (non autonomous) minimum time function subject to a non-Lipschitz dynamic, 
including a Lipschitz perturbation or Moreau's sweeping process. 

The novelty of the result lies on one hand on the lack of regularity of the velocity set of \eqref{PSP}, 
on the other on the interpretation of \eqref{PSP}, which can be given in terms of (moving) state constraints. 
In fact, in \eqref{PSP} the state constraint is accommodated into the dynamics, through the normal cone to the moving set, 
whose non-emptiness forces $x(t)$ to belong to the constraint $C(t)$. An interpretation of \eqref{PSP} 
through viability (weak flow invariance) was given by Henry \cite{He} in the case where $C(t)\equiv C$ is a convex set, 
and later generalised by Cornet to tangentially regular (also known as sleek) sets \cite{Co}, 
a case which comprises prox-regularity: actually, \eqref{PSPa} and projected differential inclusion     
\begin{equation}\label{PDI}
\left\{ \begin{array}{l}
\dot{x}(t)\,\in \,  \Pi_{T_{C}(x(t))}G(x(t))\quad \mathrm{a.e.}\; t > 0,
\\
x(0)= x_{0}\in C\;
\end{array}\right.
\end{equation}
are equivalent. Here, $\Pi_{T_{C}(x(t))}(z)$ denotes the (unique) metric projection of $z$ onto the (convex) 
tangent cone to $C$ at $x(t)$. This equivalence shows that the (internal) normal part of \eqref{PSPa} 
annihilates the component of $G$ which points outwards $C$, and this effect is obtained with 
the minimum effort with respect to the length of the added normal component.

The interpretation  of \eqref{PSPa} through \eqref{PDI} gives some insights into the behaviour of 
trajectories close to the constraint: first, a trajectory is allowed, and in some cases is 
actually forced, to slide on the boundary of $C$. This 
phenomenon shows clear differences between our dynamics with active constraints \eqref{PSPa} and other 
approaches to state constrained optimal control problems. In fact, the structure of \eqref{PDI} makes irrevelant the inward/outward pointing conditions that are often required.

The literature on constrained control problems  is vast and growing. Among it, two approaches can be recognised. 
The first one requires compatibility conditions between the dynamics and the constraint, which is usually stated 
as an inward/outward pointing condition on the velocity set $G$ with respect to the normal cone $N_{C}(\cdot)$. 
Among the many results of Hamilton-Jacobi type, we quote the earliest papers by Soner \cite{So1}, \cite{So2} (with the inward pointing 
condition), \cite{FV} (with the outward pointing condition and a discussion explaining its role in the framework of the 
approach using the invariance of the epi/hypographs that we also adopt), and \cite{FM}, \cite{FM1} which contain 
the state of the art on the subject. The inward/outward pointing condition approach makes a strong use 
of distance estimate tools (see, e.g., \cite{BeFrVi}), which are not used here. We further
observe that in our dynamics the outward pointing condition is forbidden (see \eqref{PDI}), 
while the inward one is irrelevant. On the other hand, in the approach of Soner, Frankowska and Vinter a priori 
regularity properties of the constraint $C$ play a minor role. Some regularity properties on $C$, however, follow from 
the inward pointing condition. For example, the non-emptiness of the interior of the Clarke tangent cone to $C$ is required. 

Our approach -- instead -- requires the assumption of prox-regularity of $C$. This condition is, on one hand, 
more restrictive, since in particular does not allow inward corners, on the other permits outward cusps, 
which exhibit empty Clarke tangent cone. 

The second approach is based on the fact that actual constraints are usually tame, and so both $C$ and the 
dynamics are assumed to be stratified (see \cite{BrHo}, \cite{BaWo}, \cite{HeZi}). In this approach, 
the main difficulty to be overcome is the Zeno phenomenon, namely touching a stratum on a totally 
disconnected bounded infinite set of times and patching together the Hamiltonians related to each stratum. 
Such a phenomenon does not represent a difficulty in our framework, since our Hamiltonian admits a unitary representation.

Finally, we note that Example 1 shows that the interpretation of the non autonomous dynamics \eqref{PSP_intro} through a projection of the
controlled part onto the tangent cone to the constraint is not valid. Therefore, the technique on which the results contained in \cite{Se05} are based
cannot be used in our framework.

\medskip

\noindent \textbf{Acknowledgment.} The authors wish to thank Peter Wolenski for kindly suggesting the problem and useful discussions, and 
H\'{e}l\`{e}ne Frankowska for suggesting to use the hypograph of $T$ (in substitution to the epigraph with the reversed dynamics)
under a continuity assumption. They are also indebted with the anonymous Associate Editor, for the careful reading and numerous
suggestions to improve the presentation.

\end{document}